\documentclass[12pt]{amsart} 
\usepackage[a4paper]{geometry}
\usepackage{amsthm}
\usepackage{amssymb}
\usepackage{amsmath,mathrsfs}
\usepackage[all]{xy}
\usepackage[bookmarksnumbered]{hyperref}
\usepackage{enumitem}
\usepackage{eucal}
\usepackage{ulem, undertilde}
\usepackage{setspace}
\usepackage{rotating}

\makeatletter
\newtheorem*{rep@theorem}{\rep@title}
\newcommand{\newreptheorem}[2]{%
\newenvironment{rep#1}[1]{%
 \def\rep@title{#2 \ref{##1}}%
 \begin{rep@theorem}}%
 {\end{rep@theorem}}}
\makeatother

\newtheorem{thm}{Theorem}[section]
\newreptheorem{thm}{Theorem}
\newtheorem*{thm*}{Theorem}
\newtheorem{lem}[thm]{Lemma}
\newtheorem{defn}[thm]{Definition}
\newreptheorem{defn}{Definition}
\newtheorem*{defn*}{Definition}
\newtheorem{prop}[thm]{Proposition}
\newreptheorem{prop}{Proposition}
\newtheorem{prop*}{Proposition}
\newtheorem{cor}[thm]{Corollary}
\newreptheorem{cor}{Corollary}
\newtheorem{rem}[thm]{Remark}

\newtheorem{conj}[thm]{Conjecture}

\theoremstyle{remark}

\newcommand{\Z}{\mathbb{Z}}
\newcommand{\I}{\mathbb{I}}

\newcommand{\la}{\leftarrow}
\newcommand{\ra}{\rightarrow}
\newcommand{\lra}{\longrightarrow}
\newcommand{\cofra}{\rightarrowtail}

\newcommand{\rcong}{ \begin{sideways}$\cong$\end{sideways}}
\newcommand{\Fun}{\mathrm{Fun}}
\newcommand{\Map}{\text{Map}}

\DeclareMathOperator{\Top}{Top}
\DeclareMathOperator{\Hom}{hom}
\newcommand{\join}{\ast}
\newcommand{\smsh}{\wedge}
\newcommand{\ox}{\otimes}
\newcommand{\x}{\times}
\DeclareMathOperator{\smcoprod}{\amalg}
\DeclareMathOperator{\holim}{holim}
\DeclareMathOperator{\colim}{colim}
\DeclareMathOperator{\hocolim}{hocolim}
\newcommand{\T}{\mathrm{T}}
\renewcommand{\P}{\mathrm{P}}
\newcommand{\X}{\mathcal{X}}
\newcommand{\ycal}{\mathcal{Y}}

\newcommand{\red}{\textrm{red}}
\newcommand{\nil}{\textrm{nil}}

\newcommand{\gscr}{\mathscr{G}}
\newcommand{\ccal}{\mathcal{C}}
\newcommand{\dcal}{\mathcal{D}}
\title{Goodwillie Calculus via Adjunction and LS Cocategory} 
\author{Rosona Eldred}
\email{eldred@uni-muenster.de}
\address{Fachbereich Mathematik \\
University of M\"{u}nster\\
Einsteinstr. 62\\
48149 M\"{u}nster\\
Germany
}

\begin{document}
%%%%%%%%%%%%%%%%%%%%%%%%%%%%%%%%%%%%%% 

\begin{abstract} 
{\footnotesize
In this paper, we establish a new monadic structure on the intermediate constructions, $\T_n F$, of Goodwillie's calculus of functors. We show that as a result these functors  take values in spaces of Hopkins' symmetric Lusternik-Schnirelman(LS)cocategory $\leq n$, which is an upper bound on the homotopy nilpotence class of the space. This property allows us to extend results Biedermann-Dwyer linking Goodwilie Calculus to homotopy nilpotence and of Chorny-Scherer on the vanishing of Whitehead products for spaces which are values of $n$-excisive functors. 

We also use a dual form of our adjunction to give a rigorous formulation of homotopy functor analog of McCarthy's Dual Calculus, where $n$-co-excisive functors take certain pullback cubes to pushout cubes, and dualize our results of calculus and LScocategory to dual calculus and LScategory.
%In this paper, we show that for reduced homotopy endofunctors of spaces, F, and for all $n \geq 1$ there are adjoint functors $R_n, L_n$ with $T_n F \simeq R_n F L_n$, where $P_n F$ is the $n$-excisive approximation to $F$, constructed by taking the homotopy colimit over iterations of $T_n F$. This then endows $T_n$ of the identity with the structure of a monad and the $T_n F$'s are the functor version of bimodules over that monad. It follows that each $T_n F$ (and $P_nF$) takes values in spaces of symmetric Lusternik-Schnirelman cocategory $n$, as defined by Hopkins. This also recovers recent results of Chorny-Scherer. The spaces $T_n F(X)$ are in fact classically nilpotent (in the sense of Berstein-Ganea) but not nilpotent in the sense of Biedermann and Dwyer. We extend the original constructions of dual calculus to our setting, establishing the $n$-co-excisive approximation for a functor, and dualize our constructions to obtain analogous results concerning constructions $T^n$, $P^n$,and LS category.
 }
 \end{abstract}

%%%%%%%%%%%%%%%%%%%%%%%%%%%%%%%%%%%%%% 
\maketitle
%\tableofcontents
%\setcounter{tocdepth}{2}

%%%%%%%%%%%%%%%%%%%%%%%%%%%%%%%%%%%%%% 
%%%%%%%%%%%%%%%%%%%%%%%%%%%%%%%%%%%%%% 
%
%  INTRODUCTION section
%
%
%%%%%%%%%%%%%%%%%%%%%%%%%%%%%%%%%%%%%% 
%%%%%%%%%%%%%%%%%%%%%%%%%%%%%%%%%%%%%% 
\section{Introduction}%
%%%%%%%%%%%%%%%%%%%%%%%%%%%%%%%%
This paper gives a unified treatment of two important areas of algebraic topology:  Goodwillie's calculus of functors and Lusternik-Schnirelmann(LS) cocategory, giving new insights into both fields and deepening the connection between functor calculus and homotopical nilpotency.  Goodwillie calculus provided one of the important tools for calculating algebraic K-theory and has also been used extensively to make progress in calculating stable homotopy groups of spheres.  LS category is a numerical homotopy invariant of a space $X$ roughly the minimal size of a certain kind of covering, it is a lower bound for the number of critical points of a function on $X$ and has been more recently used in formulations of topological complexity; there are a myriad of duals, all called LS cocategory.   

We establish that the functor calculus approximates a functor $F$ by one taking values in spaces of finite LS cocategory. This follows from proving that the first-order polynomial approximations $T_nF$ have a natural decomposition as $T_nF \simeq R_n F L_n$ where $(R_n, L_n)$ are an adjoint pair of functors. This allows us to rigorously formalize a dual calculus where functors are $n$-co-excisive (rather than $n$-co-additive, as in the original formulation of McCarthy) by dualizing the adjoint pair and defining $T^nF:= L^n FR^n$.

To state our results more precisely, we need to give a small overview of LS (co)category and Goodwillie's functor calculus. A more thorough background can be found in Section \ref{sec:bckgd}. 

The LS category of a space $X$, denoted LScat(X), is %a numerical homotopy invariant, 
defined as one less than the minimal number of open sets contractible in $X$ needed to cover $X$ \cite{LS}. % This was originally defined in the context of calculus of variations, for manifolds, and is a lower bound for the number of critical points of a function on $X$ \cite{LS}.  
The direct definition in terms of coverings is not very practical. %generalizing to spaces led to a development of several alternate characterizations. In particular, those of
Alternate approaches arose, and those of  Ganea\cite{GaneaLS} and later Hopkins\cite{hopkins} were of a similar form: for each space $X$, they each provided a construction $G_nX$ and natural maps $g_n: G_nX \ra X$  such that LScat(X) $\leq n$ if $g_n$ had a homotopy section.  The duals, %se constructions admitted duals, such that the LScocategory of a space $X$, 
LScocat(X), are less than or equal to $n$ when a natural map $g^n: X \ra G^nX$ has a homotopy retract. %The main definitions of LScategory have all been shown to be equivalent--see e.g. the survey in \cite[\S V Ch. 27]{felix}--however, the various kinds of LScocategory have not. %Hopkins' version, \textit{symmetric LScocat(X)}\cite[Section 3, p221-222]{hopkins}, should be thought of as the smallest number that $X$ can be written as the ``homotopy intersection" of $(n+1)$ open sets contractible in $X$.  LScategory and LScocategory of a space 
%LS(co)cat$(X)$ are generally difficult invariants to calculate,  with recently renewed interest due to applications in the study of topological complexity\cite{iwasesakai10,farberyuzvinsky04}.

The key insight which led to this paper was that  Hopkins's definition of symmetric LScocat can be equivalently stated using constructions of Goodwillie calculus. 
 
% \subsection{Goodwillie's calculus of homotopy functors}
 We refer to a  weak-equivalence preserving functor $F$ as a \textit{homotopy functor}. Let $F$ be a homotopy functor of spaces % over some fixed space $Y$. We take $Y$ to be a point. Our spaces will be 
 that are at least 0-connected. Excisive functors take homotopy pushout squares to homotopy pullback squares and $n$-excisiveness is a generalization of this in terms of $(n+1)$-cubical diagrams. Goodwillie's calculus then provides a tower of $n$-excisive approximations to $F$, $P_n F$, for $n \geq 0$.
\[
\cdots \ra \P_n F \ra \P_{n-1} F\ra \cdots \ra \P_1 F \ra \P_0 F,
\]
whose homotopy limit is denoted $\P_\infty F$. 
 
We denote by $\mathscr{P}([n])$ the power set on $[n]=\{0, \ldots, n\}$ and $\mathscr{P}_0([n])$ the nonempty subsets.
For a homotopy functor $F$, let $\T_n F(X):= \holim_{U \in  \mathscr{P}_0 ([n])} F(U \join X)$; here, $\ast$ denotes topological join. There is a natural map $t_n: F(X) \ra \T_n F(X)$ induced by inclusion of the empty set.  Then  $\P_n F$ is defined as the homotopy colimit over a directed system of these finite homotopy limit constructions along the iterated  $t_n$'s:
%\[
%\[
%
$\P_n F(X):= \hocolim (\T_n F(X) \ra \T_n^2 F(X) \ra \cdots).$ When $n=1$, $T_1F \simeq \Omega F \Sigma$ and $P_1 F \simeq \Omega^\infty F \Sigma^\infty$. 
%\]

\subsection{Statements of results} 
In this language, we have the following reformulation, where $\mathbb{I}$ is the identity functor of spaces. %-- for the original definition and how to translate between the two, see section \ref{sec:hopkinsLS}: 
%%%%%%%%
\begin{prop} \label{def:LScocatTn}
A space $X$ has symmetric-LScocategory less than or equal to $n$ if and only if $X$ is a homotopy retract of $\T_n \I (X)$. 
\end{prop}
%%%%%%%%

That is, the $\T_n\I$'s should be seen as the classifying objects of this property of having symmetric LScocat $\leq n$.  Moreover, we establish the following, which are corollaries of Theorem \ref{thm:main}. The proofs will be given in section \ref{sec:lscocatproofs}.

%%%%%%%%
\begin{repcor}{cor:sym}
For each $n \geq 1$, the functor $\T_nF$ takes values in spaces of symmetric LS cocat $\leq n$, as do the $\P_nF$.   
\end{repcor}
%%%%%%%%

A space has Whitehead length $n$ if, for all $(n+1)$-tuples of elements in $\pi_*X$, $(x_0,x_2,\ldots, x_n)$, their iterated Whitehead product, denoted $[x_1,[x_2,[\cdots x_n]\cdots ]$, vanishes. 

%%%%%%%%
\begin{repcor}{cor:WL}
For every space $X$ and $n\geq 1$, the spaces $\T_nF(X)$ and $\P_n F(X)$ have Whitehead length $n$.
\end{repcor}
%%%%%%%%

%This is a consequence of the following Theorem.  
We refer to a functor $F$ as \textit{strongly reduced} when $F$ of a point is a point and use Goodwillie's terminology of \textit{reduced} for when $F$ of a point is homotopic to a point.  Functors which are strongly reduced include those which are basepoint-preserving homotopy functors of based, connected topological spaces, the subject of study in \cite{Bied-Dwy}. 

%%%%%%%%
\begin{repthm}{thm:main}
Let $F$ be a reduced homotopy endofunctor of topological spaces. For all $n\geq1$, there are adjoint functors  $R_n, L_n$  such that $\T_n F$ is weakly equivalent to $R_n F L_n$. For $F$ strongly reduced, we have that $\T_n F=R_n F L_n$. In particular, $\T_n \I = R_nL_n$, i.e. has the structure of a monad for all $n$. 
\end{repthm}
%%%%%%%%

To state our adjunctions, we need two categories of functors, one of punctured cubical diagrams of spaces and the subcategory of \textit{reduced} punctured cubical diagrams of spaces. We say $\X \in$ Fun$(\mathscr{P}_0 ([n]), \Top) = \Top^{\mathscr{P}_0 ([n])}$ is \textit{reduced} if $\X(S)$ is a point, for all $S$ with $|S|=1$. We denote the subcategory of reduced $\mathscr{P}_0([n])$-diagrams in spaces by $\widetilde{\mathrm{Fun}}(\mathscr{P}_0([n]), \Top)$.

We use a common model for the homotopy limit of a $\mathscr{P}_0 ([n])$-diagram $\X$ in $\Top$ (see section \ref{sec:bckgd} for details) whose natural left adjoint takes a space $Y$ to the cube $S\mapsto Y \x \Delta^S$, where $\Delta^S := \Delta^{|S|-1}$. 

The adjoint pair of the above theorem is then given by 
\[
L_n := \text{red} \circ (S\mapsto - \x \Delta^S )\text{  and  }R_n := \holim \circ \text{ inc};
\]  
The functor \text{red} takes an unreduced diagram to one where the singelton-indexed spaces are forced to be points by a colimit construction which then propogates this change through the rest of the diagram, collapsing the images of the singleton-indexed spaces. Here, inc is inclusion of a subcategory. We will elaborate on these in section \ref{sec:genproof}.

We follow the normal convention of left adjoints being written as the top arrow of each pair:
\[
L_n:
\xymatrix{
\Top \ar@<+5 pt>[rr]^-{S\mapsto -\x \Delta^S }&& \ar[ll]^-{\holim} \Fun (\mathscr{P}_0([n]), \Top) \ar@<+5 pt>[r]^{\text{red}} & \ar[l]^{\text{inc}}\widetilde{\Fun} (\mathscr{P}_0([n]), \Top) \\
}
:R_n
\]

\noindent We also establish that the adjunction is a Quillen adjunction. 
%We now discuss other corollaries of Theorem \ref{thm:main}. 
%
%
%
%Given a pair of elements $(x_1, x_2) \in \pi_j X \x \pi_k X$, their Whitehead product is the following composite in $\pi_{j+k-1} X$,  with $\epsilon$ the attaching map used to obtain $S^j \x S^k$ by attaching a $(j+k)$-cell to the wedge sum:
%\[
%[x_1, x_2]: \xymatrix{S^{j+k-1} \ar[r]^\epsilon& S^j\Wedge S^k\ar[r]^-{x_1\Wedge x_2} & X}
%\]
%We can consider the iterated Whitehead product of more than 2 elements, by taking the Whitehead product of the first two in the list and the Whitehead product of the result with the next, e.g. for $(x_1,x_2,x_3)$, we take $[[x_1,x_2], x_3]$ and so on for larger tuples.  

We take $R^n:=R_n^{op}$, hocolim precomposed with an inclusion, and present its adjoint, which we will call $L^n$. 

We use the category $\Fun(\mathscr{P}^1([n]), \Top)$ of copunctured (with the final space removed) diagrams of spaces and the full subcategory  $\utilde{\Fun}(\mathscr{P}^1([n]), \Top)$ of \textit{co-reduced} copunctured cubical diagrams of spaces. That is, each diagram $\X \in \utilde{\Fun}(\mathscr{P}^1([n]), \Top)$ has the property that $\X([n]-S)$ is contractible, for $|S|=1$.  Dual to the previous case, there is a co-reduction functor, \textit{cor}, that takes a copunctured cube to a coreduced one and we can send $X$ to a co-punctured diagram of spaces via $X \mapsto ([n]-S \mapsto X^{\Delta^S})$. Then, we will prove in section \ref{sec:dualproofs} that  \\

 %%%%%%%%
 \begin{repprop}{prop:dualadj}
 For each $n \geq 1$, there are adjunctions between these categories as follows:
\[
L^n:
\xymatrix{
\utilde{\Fun}(\mathscr{P}^1 ([n]), \Top) \ar@<+3 pt>[r]^-{inc}& \ar@<+3pt>[l]^{cor}  \Fun (\mathscr{P}^1([n]), \Top) \ar@<+3 pt>[rr]^-{\hocolim} & &\ar@<+3 pt>[ll]^-{X \mapsto ([n]-S \mapsto X^{\Delta^S})} \Top \\
}:R^n
\]
 \end{repprop}
  %%%%%%%%

LScategory is a better understood invariant of a space than its dual, LScocategory, and  would naturally tie in to a dual formulation of Goodwillie's calculus. One can consider the Eckmann-Hilton dual theory: co-excisive functors take homotopy pullback squares to homotopy pushout squares and higher n-co-excisive functors satisfying a more general condition involving $(n+1)$-cubes.  An n-co-additive version of this was first developed by Randy McCarthy \cite{Dual}. A dual tower for a functor $F$ has stage which are functors that have natural transformations \textit{to} $F$. %Some functors like $K$-theory most naturally map to other functors rather than receive maps, e.g. the trace map from $K$-theory to TC.  
McCarthy originally constructed the dual calculus as a way to gain an approximation to $K$ theory with a natural transformation into it.  %hich mapped into it. He did this by dualizing his cotriple calculus of \cite{JM-cotriples}.
 
Since the dual calculus \cite{Dual} was defined before only in a (co)triple way, there was, before this paper, no $\T^n F$ which one iterates to produce $\P^nF$, the $n$-co-excisive approximation to a functor. We use our dual adjunction $R^n, L^n$ and define $\T^n F$ in an analogous way as for $\T_n F$.

%%%%%%%%
\begin{repdefn}{defn:dualadj}
%\textit{
Given our adjoint pair $R^n, L^n$, for all $n\geq1$, we define $\T^n F:= L^n FR^n$; $T^n \I$ is then the comonad $L^n R^n$. There is a natural map $t^nF: \T^n F \ra F$ which is the map from a hocolimit of a co-punctured diagram to its final entry. 
\end{repdefn}
%%%%%%%%

In this language, we may re-state another of Hopkins's definitions as 

%%%%%%%%
\begin{repprop}{defn:symcat}
For a space $X$, symmetric LS cat(X) $\leq n$ if and only if  the natural map $\T^n\I (X)\ra X$ has a section up to homotopy.
\end{repprop}
%%%%%%%%

Most of the Eckmann-Hilton duals of the previous results listed hold, with Whitehead product replaced by cup product, LScocat replaced by LScat, and calculus constructions replaced by their duals. We provide the statements and their proofs in sections \ref{sec:dualproofs} and \ref{sec:lscat} and further discussion of the dual calculus in section \ref{sec:bkDGC}.

We hope that this connection between Goodwillie calculus and LScategory will lead to a tangible connection between calculus of variations-- the field from which the notion of LScategory originates-- and Goodwillie calculus. 

\newpage

%%%%%%%%%%%%%%%%%%%%%%%%%%%%%%%%%%%%%%%
\subsection{Conjectures and partial results} 
%%%%%%%%%%%%%%%%%%%%%%%%%%%%%%%%%%%%%%%
Theorem \ref{thm:main} has a logical conjecture following it, which we believe to be true (and are working towards establishing): 

\begin{conj}
There is an appropriate category $\ccal$ which each  $\widetilde{\Fun} (\mathscr{P}_0([n]), \Top)$ maps to and adjunctions
\[
L_n^\infty: 
\xymatrix{
\Top \ar@<+5 pt>[r] & \ar[l] \ccal 
}
:R_n^\infty 
\]
 for each $n$ , such that $P_nF \sim R_n^\infty F L_n^\infty$. 
\end{conj}
\noindent Here $R_n^\infty, L_n^\infty$ should relate to $R_n, L_n$ in the way $\Omega^\infty, \Sigma^\infty$ relate to $\Omega, \Sigma$.  This additionally would mean that whenever $F$ is a monad, $P_nF$ is a monad for all $n$.

To state the next corollary, we need a few definitions involving notions of nilpotence. % of a space.
%, first. For a group $G$ and elements $g,h, k$, we say that the commutator $[g,h]:= ghg^{-1}h^{-1}$ is of length 2, $[[g,h],k]$ is of length 3, and similar for higher lengths. Then 
A group $G$ is \textit{nilpotent of class  $\leq n$} when all commutators of length $> n$ are trivial. Nilpotence class 1 is abelian; nilpotence class is a measure of how close a group is to abelian. %
%
%Extending this notion to spaces lead to two notions-- first,  a 
%A space $X$ has $\nil(\pi_1 X) \leq n$ when $\pi_1 X$ is a group of nilpotence class $\leq n$. Examples of spaces with $\nil (\pi_1 X) \leq 1$ include $H$-spaces.  
Note that each loop in $\Omega X$ has a homotopy inverse, so we can consider length n \textit{homotopy} commutators of loops. We say that $\nil (\Omega X) \leq n$ if all homotopy commutators of length $>n$ are nullhomotopic.  Examples of spaces with $\nil (\Omega X) \leq 1$ are two-fold loopspaces. %We note that $\nil (\Omega X)\leq n$ implies $\nil (\pi_1 X) \leq n$. 

Goodwillie calculus has played a role in providing another notion of nilpotence for spaces, which we extend in this paper. Biedermann and Dwyer \cite{Bied-Dwy} constructed Lawvere-style theories $\gscr_n$  for each $n$ from $\Omega P_n \I$. They showed that $\Omega P_n F$ take values in spaces which are homotopy algebras over these theories; they called these homotopy algebras \textit{homotopy $n$-nilpotent groups}. They assert that the values of functors of the form $\Omega P_n F$ are the only homotopy algebras over $\gscr_n$; the proof is left to \cite{unBiedDwy}. The homotopy $1$-nilpotent groups are infinite loopspaces. If $X$ is a homotopy $n$-nilpotent group in this sense, we will say that $\gscr_n\nil (X) \leq n$.  This is an upper bound on $\nil (\Omega X)$.  Being a homotopy $n$-nilpotent group  is much more than a number, but we do not make use of the extra structure.

%%%%%%%%
\begin{conj}\label{cor:nil-not-BD}
For $F$ such that $\T_n F$ is not $n$-excisive and for each $X$, $\T_nF$ takes values in spaces which are classically nilpotent  but not nilpotent in the sense of Biedermann and Dwyer. That is, $\nil (\Omega \T_n FX) \leq n$ but $\mathscr{G}_n\nil (\T_n FX)$ is not. 
\end{conj}
%%%%%%%%

%%%%%%%%%
%\noindent \textbf{Conjecture \ref{cor:nil-not-BD}}:
%\textit{For $F$ such that $\T_n F$ is not $n$-excisive and for each $X$, $\T_nF$ takes values in spaces which are classically nilpotent  but not nilpotent in the sense of Biedermann and Dwyer. }\\ %That is, $\nil (\Omega \T_n FX) \leq n$ but $\mathscr{G}_n\nil (\T_n FX)$ is not. }
%%%%%%%%

%\textit{Evidence} 
Classical nilpotence follows from the inequalities of (\ref{ineq-nil}) and the fact that $\T_n F$ take values in spaces of symm LScocat $\leq n$.
There is an equivalence of categories \cite{unBiedDwy} %between 
\[
\{\textrm{values of functors } \Omega F, \; F\; n-\textrm{excisive} \} \sim \{\textrm{homotopy $n$-nilpotent groups}\}.
\]

\noindent Under our hypotheses,
%
%\noindent Since 
$\T_n F$ is not $n$-excisive unless $F$ is (i.e. unless it equals $\P_n F$),  which prevents
$\Omega \T_n FX$ from being equivalent to $\Omega \P_nF X$ for $X$ any space.  %This is not enough to show that $\Omega \T_n FX$ will not be a homotopy $n$-nilpotent group (using the above equivalence of categories) in the sense of Biedermann and Dwyer. 

The trouble at this point (which was communicated to me by Clemens Berger) is that this does  not guarantee that there does not exist, for each $X$, an $n$-excisive functor $G$ and space $Y$ such that $\Omega \T_n F X \simeq \Omega \P_n FY$. 
%\end{proof}

In answer to the question of a reviewer about the original hypotheses of Conjecture \ref{cor:nil-not-BD}, we also establish the following result. The proof is 
in Section \ref{sec:lscocat}. % , after the proof of Corollary \ref{cor:nil-not-BD}. 
For a thorough definition of analyticity, see Section \ref{sec:bckgd}.

%%%%%%%%
\begin{repprop}{prop:radconv}
Let F be a $\rho$-analytic homotopy endofunctor of spaces for some $\rho\geq 0$. Assume that there is also an $n > 0$ such that $\T_nF$ is $n$-excisive. Then $F(X)$ is weakly equivalent to $\P_n F(X)$ for all $X$ of connectivity $\geq \rho$. 
 \end{repprop}
 %%%%%%%%
 %%%%%%%%

%%%%%%%%%%%%%%%%%%%%%%%%%%%%%%%%%%%%%%%%
%\subsection{Implications and future work} 
%%%%%%%%%%%%%%%%%%%%%%%%%%%%%%%%%%%%%%%%
%There are maps not just between the iterations $\T_nF \ra \T_n^2F$ but also $\T_nF \ra \T_{n-1}F$ and moreover, $\T_n^kF \ra \T_{n-1}^kF$ for all $k \geq 1$. These then form, for each $k$, a tower and we have a diagram-- a directed systems of towers. %-- as in Figure \ref{fig:towers} below. 
% 
%The first consequences of $\T_nF$ taking values in spaces of LScocat $\leq n$ includes that the 0th row or $k=1$ tower already takes values in spaces which satisfy classical nilpotence. That is, $\nil(\Omega \T_n FX) \leq n$ for all $X$, for all $n \geq 1$. 
%
%We should then view the higher iterations, the $k$-towers for arbitrary $k$, as nilpotent-type filtrations becoming homotopically better as we iterate.  
%
%
%
%Here, homotopically better should correspond to vanishing of higher invariants -- the spectral sequence coming from the first stage should then have differentials which relate to the vanishing of Whitehead products. The spectral sequence for each higher stage should then bear the same relation to these higher invariants. One question which remains is what these generalized Whitehead-product-type invariants are, especially what they converge to for the $P_n F$'s. 
%

%%%%%%%%%%%%%%%%%%%%%%%%%%%%%%%%%%%%%% 
\subsection{Related work} 
%%%%%%%%%%%%%%%%%%%%%%%%%%%%%%%%%%%%%% 
The vanishing of iterated Whitehead products for values of $\Omega \P_n F$  (part of Corollary \ref{cor:WL}) can also be seen by direct computations of Chorny and Scherer \cite[Theorem 2.1]{chorny-scherer}. 
Related to Corollary \ref{cor:sym}--the author has become aware that independently, Christina Costoya, J\'{e}r\^{o}me Scherer and Antonio Viruel have shown that the $\P_n F$'s take values in spaces with \textit{inductive} (i.e. Ganea) cocategory $n$ \cite{csv}. We recover this result by combining Cor \ref{cor:sym} and the inequality that for a space $X$, inductive LScocat(X) is less than or equal to symmetric LScocat(X) of \cite{hopkins}.
We also point out that the constructions of Ganea used in the definition of inductive LS category were proven by Deligiannis to have the structure of comonads \cite{comonads}. Our proof is necessarily significantly different than a dualization of this result, as we lack an inductive definition of the $\T_n$'s. 

%%%%%%%%%%%%%%%%%%%%%%%%%%%%%%%%%%%%%% 
\subsection{Organization} 
%%%%%%%%%%%%%%%%%%%%%%%%%%%%%%%%%%%%%% 
The paper is organized as follows. Section \ref{sec:bckgd} provides background on cubical diagrams, ho(co)limits, Goodwillie calculus and structures such as being a left/right M-functor for M a monad.   There are two main blocks of results, one for the usual calculus (section \ref{sec:proofs}) and one for the dual calculus (sections \ref{sec:dualproofs} and \ref{sec:bkDGC} ), each followed by a section ( \ref{sec:lscocat} and \ref{sec:lscat}, respectively) explaining the relationship with LS(co)category and giving proofs of the relevant corollaries. 

%\if false
%%%%%%%%%%%%%%%%%%%%%%%%%%%%%%%%%%%%%% 
%%%%%%%%%%%%%%%%%%%%%%%%%%%%%%%%%%%%%% 
\subsection{Acknowledgements}
The author would like to thank Tom Goodwillie for help in simplifying the construction of the adjoint functors, and Bill Dwyer for many discussions about nilpotence and calculus of functors. 
%%%%%%%%%%%%%%%%%%%%%%%%%%%%%%%%%%%%%% 
%%%%%%%%%%%%%%%%%%%%%%%%%%%%%%%%%%%%%% 

%%%%%%%%%%%%%%%%%%%%%%%%%%%%%%%%%%%%%% 
%%%%%%%%%%%%%%%%%%%%%%%%%%%%%%%%%%%%%% 
%
%  Background section
%
%
%%%%%%%%%%%%%%%%%%%%%%%%%%%%%%%%%%%%%% 
%%%%%%%%%%%%%%%%%%%%%%%%%%%%%%%%%%%%%% 
\section{Background}\label{sec:bckgd}
This section contains a variety of information useful for non-experts. We first introduce terminology about cubical diagrams in section \ref{sec:cubes} and descriptions of our models for ho(co)lim of (co)punctured cubes in section \ref{sec:ho(co)lim}. These are necessary for the following constructions of Goodwillie calculus in section \ref{sec:bkGC}. %We follow with an overview in section  \ref{sec:bkGCnilp} of the relationship (thus far) between Goodwillie calculus and nilpotence of a space. Shifting gears, 
We introduce the definition of a monad $M$, as well as left/right $M$-modules and the functor analog, left/right $M$-functors in section \ref{sec:bkMod}.We leave a discussion of the dual calculus for section \ref{sec:bkDGC}. %We provide in section \ref{sec:QFunc} a definition of a Quillen adjunction and a proposition which we find useful in our proofs later.  %We then say a little about LS cocategory in section \ref{sec:bkLS}.

%\todo[color=cyan, inline]{Appendix, or future section?}

%%%%%%%%%%%%%%%%%%%%%%%%%%%%%%%%%%%%%%%%
% Cubes
%%%%%%%%%%%%%%%%%%%%%%%%%%%%%%%%%%%%%%%%
\subsection{Cubes and cubical diagrams}\label{sec:cubes}
%%%%%%%%%%%%%%%%%%%%%%%%%%%%%%%%%%%%%%%%

We take $\Delta$ to be the category of finite ordered sets and monotone maps, with elements $[n] =\{0,1, \ldots n\}$. If $S$ is a finite set, we denote by $\Delta^S$ the topological simplicial complex $\Delta^{|S|-1}$, so that $\Delta^{[n]} = \Delta^{n}$. We denote by $\mathscr{P}(S)$ the power set of the set $S$, which we will freely use to also mean the corresponding category with morphisms given by inclusion and objects the subsets of $S$. We can also use $\mathscr{P}(S)$  to mean its diagrammatic representation; the following is a diagrammatic representation of the category $\mathscr{P} ([1])$. 
\[
\xymatrix{
\emptyset \ar[r]\ar[d] & \{0\}\ar[d]\\
\{1\} \ar[r] & \{0, 1\}\\
}
\]

We will denote by $\mathscr{P}_0(S)$ the subcategory without the emptyset and $\mathscr{P}^1 (S)$ the subcategory of $\mathscr{P}(S)$ with $S$ removed.  %
An \textit{$(n+1)$-cube} of spaces is then a functor from $\mathscr{P}([n])$ to spaces, with sub-diagrams given by restricting to $\mathscr{P}_0([n])$ or $\mathscr{P}^1([n])$, the punctured or co-punctured $(n+1)$-cube, respectively.  %
For $\mathcal{X}$ a $\mathscr{P}^1([n])$-diagram, rather than index $\mathcal{X}$ by the subsets $S \in \mathscr{P}^1([n])$, it is customary to consider instead $\mathcal{X}([n]-U)$, where $U \in \mathscr{P}_0 ([n])$.
%\ldots

%We will commonly refer to a homotopy pullback square as \textit{cartesian} and a homotopy pushout square as \textit{cocartesian}.  
An $n$-cube $\mathcal{X}$ is \textit{cartesian} if its initial point, $\mathcal{X}(\emptyset)$,  is weakly equivalent along the natural map to the homotopy limit of the rest of the diagram, i.e. if $\mathcal{X}(\emptyset) \overset{\sim}{\ra} \holim_{U \in \mathscr{P}_0([n])} \mathcal{X}(U)$.  We say that $\mathcal{X}$ is \textit{cocartesian} if $\mathcal{X}([n])\overset{\sim}{\la}\hocolim_{U \in \mathscr{P}_0([n])} \mathcal{X}([n]-U)$ is a weak homotopy equivalence.  The terms \textit{strongly cocartesian} and \textit{strongly cartesian} are used to denote that every sub-2-face (i.e. every sub-square) is \textit{cocartesian} (or, respectively, \textit{cartesian}). We note that one may add ``homotopy" before (co)cartesian in the preceding if we were going to distinguish between homotopy (co)limits and (co)limits, but we omit that modifier in keeping with conventions used in work of Goodwillie\cite[Definition 1.3]{GC2}.

%%%%%%%%%%%%%%%%%%%%%%%%%%%%%%%%%%%%%%%%
% Cubes
%%%%%%%%%%%%%%%%%%%%%%%%%%%%%%%%%%%%%%%%
\subsection{Ho(co)lim for n-cubes}\label{sec:ho(co)lim}
%%%%%%%%%%%%%%%%%%%%%%%%%%%%%%%%%%%%%%%%
For a punctured cube of spaces, $\X$, the model for homotopy limit we use $\Hom_{\Top^{\mathscr{P}_0([n])}}(\Delta^S|_{S\in \mathscr{P}_0([n])},\X)$. For $\X$ a punctured square, an element of this hom-space is a map from the left diagram to the right diagram:
\[
\xymatrix{
   				   & \Delta^0 \ar[d]^{d^1}\\
   \Delta^0 \ar[r]_{d^0} & \Delta^1\\
}
\hspace{ 1 in}
\xymatrix{
   & \X(0) \ar[d]^f\\
   \X (1) \ar[r]_g & \X (\{0,1\})\\
}
\]
 where $\Delta^i$ are topological simplices (the realizations of $\Delta^i$, by common abuse of notation) and the $d^i$ are the coface maps including the point at one end or the other of the interval, induced by the two inclusions of $\{0\}$ into $\{0,1\}$ as $0$ or $1$. An element of this hom-space is a tuple $(x_0, x_1, \gamma) \in \X(0) \x \X(1) \x \X(\{0,1\})^I$ such that the path $\gamma$ in $ \X (\{0,1\})$ has $\gamma (0) = f (x_0)$ and $\gamma (1) = g(x_1)$. 

This holim model has a natural left adjoint, which takes a space $X$ and sends it to the punctured cubical diagram $S \mapsto X \x \Delta^S$ (see Example 8.13 of \cite{duggerPrimer}).  This is half of the left adjunction in Theorem \ref{thm:main}.

For a small category $\dcal$ and the category of $\dcal$-diagrams in another category $\ccal$, the opposite category of $\dcal$ diagrams in $\ccal$ is the category of $\dcal^{op}$ diagrams in $\ccal$:   $(\ccal^\dcal)^{op} = \ccal^{\dcal^{op}}$. For example, $\dcal= \Delta$, then $\ccal^\dcal$ is cosimplicial objects in $\ccal$ and the opposite category is simplicial objects in $\ccal$. In our setting, the opposite category of punctured cubes of spaces, denoted $\Fun (\mathscr{P}_0 ([n]), \Top)$ is the category of copunctured cubes of spaces--cubical diagrams with the final object missing -- denoted  $\Fun (\mathscr{P}^1 ([n]), \Top)$.

For a given a model of a holim in a category $\ccal$, one model %obtains a dual model 
for the hocolim in $\ccal^{op}$ is simply the opposite of the holim from $\ccal$.  We describe the analogous dual model for a homotopy colimit, following Dugger\cite[Section 8.10]{duggerPrimer}, which is a kind of tensor product.  

Given two diagrams, $X, Y$ where $X: I \ra \Top$ and $Y: I^{op} \ra \Top$, the tensor product of diagrams $X \ox Y$ is defined as the coend 
\[
X \ox Y = \text{coeq} \left[\xymatrix{ \coprod_{i \ra j} X_i \x Y_j \ar@<-5pt>[r]\ar@<5pt>[r]& \coprod_i X_i \x Y_i }\right]
\]

For a category $C$, we use $BC^{op}$ to denote the geometric realization of the nerve of its opposite category. For each $c\in C$, there is an associated undercategory $(c \downarrow C)$ of objects in C with arrows from $c$; the assignment $c \mapsto (c\downarrow C)$ is functorial.  

Then let $B(- \downarrow I)^{op} : I^{op}\ra \Top$ be the functor sending $i$ to the  $i \ra B(i \downarrow I)^{op}$.  
Then,for $X$ an $I$-diagram, a model for hocolim$_I X$  is  $X \ox B(- \downarrow I)^{op}$,  as in Example 8.12 of \cite{duggerPrimer}. 

For any copunctured diagram $\X : \mathscr{P}^1 ([n]) \ra \Top$, we form its homotopy colimit by tensoring with the diagram $S \mapsto \Delta^S$; 
this is $B(- \downarrow I)^{op}$ in our setting. 

%%%%%%%%%%%%%%%%%%%%%%%%%%%%%%%%%%%%%%%%
% Goodwillie Calc
%%%%%%%%%%%%%%%%%%%%%%%%%%%%%%%%%%%%%%%%
\subsection{Goodwillie Calculus}\label{sec:bkGC} 
%%%%%%%%%%%%%%%%%%%%%%%%%%%%%%%%%%%%%%%%
Not much background in Goodwillie calculus is needed to understand our results. %Background on 
Information regarding the dual calculus may be found in section \ref{sec:bkDGC}.  In this paper, we restrict ourselves to calculus ``over a point". Goodwillie has defined the theory more arbitrarily, for spaces over an arbitrary fixed space, and recent work has extended another model of the calculus to functors of spaces with maps factoring a fixed map (e.g. for a map $f:A \ra B$, a factorization is then a space $X$ with maps $\alpha, \beta$ such that $\beta \circ \alpha =f$), see \cite{BJM14}.  It is possible that these more general forms of calculus then give altered versions of LScocategory. That is, we expect that $\T_n \I$ for functors over arbitrary $Y$ classifies a sort of relative or fiberwise LScocategory.   

We will assume our domain and codomain are topological spaces over a point, outside of the dual calculus setting, where to talk about $P^n$, we work stably, and use functors of (Bousfield-Friedlander) spectra.

%%%%%%%%%%%%%%%%%%%%%%%%%%%%%%%%%%%%%%
%%%%%%%%%%%%%%%%%%%%%%%%%%%%%%%%%%%%%%
\subsubsection{Definitions and constructions}
%%%%%%%%%%%%%%%%%%%%%%%%%%%%%%%%%%%%%%
%%%%%%%%%%%%%%%%%%%%%%%%%%%%%%%%%%%%%%
In \cite{GC1, GC2}, Goodwillie establishes the following definition, in analogy with a function being polynomial of degree 1 or $n$: 

\begin{defn}
A functor $F$ is \emph{excisive} (i.e. 1-excisive) if it takes cocartesian squares to cartesian squares and \emph{$n$-excisive} if it takes strongly cocartesian $(n+1)$-cubes to cartesian ones.
\end{defn}

Generalized, reduced homology theories, viewed as functors  $X \mapsto \Omega^\infty (\Sigma^\infty X \smsh E)$ for some spectrum $E$,  are nice excisive functors. In particular, a functor $F$ is excisive, reduced, and preserves filtered colimits if and only if it is a generalized, reduced homology theory in this sense.

We will now give the constructions necessary to produce the $n$-excisive approximations to a functor $F$, $\P_n F$, which are assembled from finite limit constructions, $\T_n F$. We let $\join$ denote the topological join over a point. 

We recall the following definition:
%\[
$\T_n F(X) := \holim_{U \in \mathscr{P}_0 ([n])} F(U \join X).$
%\]
%
We have a natural transformation $t_n: F(X) \ra \T_n F(X)$, given by the natural map

\[
F(X)=F(\emptyset \join X) \ra \underset{U \in \mathscr{P}_0 ([n])}{\holim} ( U \mapsto F(U \join X)).
\]
That is, the map from the initial object of the square, $F(X)$, to the homotopy pullback of the rest, $\T_n F(X)$. We can take $\T_n$ of $\T_n F$, and also have the same natural transformation from initial to homotopy pullback, now $\T_n F(X) \ra \T_n (\T_n F(X)) =: \T_n^2 F(X)$. For $n=1$, see Figure \ref{fig:T_1^2}.

The degree $n$ polynomial approximation to $F$, $\P_n F$, is constructed as the homotopy colimit along iterations of $t_n$, 
\[
\P_n F(X) := \hocolim ( \T_n F(X)\ra \T_n^2 F(X) \ra \cdots).
\]

It is not immediately obvious that $\P_n F$ is in fact n-excisive and universal (up to homotopy). We refer the reader to \cite{GC1, GC3} for the details, especially Lemma 1.9 of \cite{GC3}, with alternate proof provided by Charles Rezk \cite{rezk}.

%%%%%%%%%%%%%%%%%%%%%%%%%%%%%%
\begin{figure}[h!]
\[
\scalebox{.75}{$
\begin{array}{rcl}
%%%%%%%%%%%
% FirstEquality
%%%%%%%%%%%
\T_1^2 F(X) &:= &   \holim %_{\mathscr{P}_0([1])} 
\scalebox{.90}{$\left( 
	\begin{array}{c}
	\xymatrix{
	 	& \T_1 F(\{0\}\join X) \ar[d]\\
	 \T_1 F(\{1\} \join X) \ar[r] & \T_1 F(\{0,1\} \join X)\\
	}
	\end{array}
\right)$}
\\
\\
\\
%%%%%%%%%%%
% Third Equality
%%%%%%%%%%%
&\simeq&  %\underset{\mathscr{P}_0([1])\x \mathscr{P}_0([1])} {\holim}
\holim 
\left(
	\begin{array}{ccc}
	& &%%BLOCK%%
		\scalebox{.75}{$\left( 
		\begin{array}{c}
		\xymatrix{
								&F(\{0\}\join \{0\} \join X)\ar[d]\\
		F(\{0\}\join \{1\} \join X) \ar[r] & F(\{0\}\join\{0,1\} \join X)\\
		}
		\end{array}
		\right)$}
	%%END BLOCK%%
		\\
		&&\downarrow\\
	%%BLOCK%%	
	\scalebox{.75}{$\left( 
		\begin{array}{c}
		\xymatrix{
								&F(\{1\}\join \{0\} \join X)\ar[d]\\
		F(\{1\}\join \{1\} \join X) \ar[r] & F(\{1\}\join\{0,1\} \join X)\\
		}
		\end{array}
		\right)$}
	%%END BLOCK%%	
		&\rightarrow&
	%%BLOCK%%	
	\scalebox{.75}{$\left( 
		\begin{array}{c}
		\xymatrix{
								&F(\{0,1\}\join \{0\} \join X)\ar[d]\\
		F(\{0,1\}\join \{1\} \join X) \ar[r] & F(\{0,1\}\join\{0,1\} \join X)\\
		}
		\end{array}
		\right)$}
	%%END BLOCK%%	
		\\		
	\end{array}
	\right)
\end{array}
$}
\]
\caption{$\T_1^2F(X)$}
\label{fig:T_1^2}
\end{figure}
%%%%%%%%%%%%%%%%%%%%%%%%%%%%%%

%
%%%%%%%%%%%%%%%%%%%%%%%%%%%%%%%%%%%%%%
%%%%%%%%%%%%%%%%%%%%%%%%%%%%%%%%%%%%%%
%\subsection{Taylor Tower}
%%%%%%%%%%%%%%%%%%%%%%%%%%%%%%%%%%%%%%
%%%%%%%%%%%%%%%%%%%%%%%%%%%%%%%%%%%%%% 
The collection of polynomial approximations to a functor $F, \{ \P_n F\}_{n \geq 0}$, comes with natural fibrations $\P_nF(X) \ra \P_{n-1} F(X)$ for all $n \geq 1$.

With these maps we form a tower, the Goodwillie (Taylor) tower of $F(X)$:
\[
\cdots \ra \P_n F(X) \ra \P_{n-1} F(X) \ra \cdots  \longrightarrow \P_1 F(X) \ra \P_0 F(X).
\]
Since we are restricting ourselves in this work to calculus over a point, $P_0 F(X) = F(\ast)$; in general, $P_0 F(X) = F(Y)$ for whatever space $Y$ we ware working over.

We denote by $\P_\infty F(X)$ the homotopy inverse limit of this tower.

%%%%%%%%%%%%%%%%%%%%%%%%%%%%%%%%%%%%%%%%
\subsubsection{Analyticity and convergence} 
%%%%%%%%%%%%%%%%%%%%%%%%%%%%%%%%%%%%%%%%
Heuristically, we say that a functor $F$ is $\rho$-analytic if its failure to be $n$-excisive for all $n$ is bounded with a bound depending on $\rho$; $\rho$-analytic implies $(\rho+1)$ analytic, which is a weaker condition. This gives rise to a notion of ``radius of convergence of a functor''.

More precisely, $F$ is a $\rho$-analytic functor when there exists a $q$ such that for all $n$, $F$ takes a strongly co-Cartesian $(n+1)$ cube $X$ (with connectivities of the maps $X(\emptyset) \ra \X(s)$ $k_s > \rho$) to a cube which is $(n\rho-q + \Sigma k_s)$-cartesian. That is, the map $F(\X(\emptyset))$ to the homotopy limit of the rest of the cube is $(n\rho-q + \Sigma k_s)$-connected (see definition 4.1 and 4.2 of \cite{GC2}).  This is the bound on the failure of the target cube to be cartesian, i.e. the bound on the failure of $F$ to be $n$-excisive for all $n$. 

%%%%%%%%%%%%%%%%
\begin{prop}\cite[Theorem 1.13]{GC3} \label{prop:analyt}
If $F$ is at least $\rho$-analytic and $X$ is $k$-connected for $k$ at least $\rho$ (i.e. if $X$ is in the ``radius of convergence'' of $F$) , then $F(X) \overset{\ra}{\simeq} \P_\infty (X)$.  
\end{prop}
%%%%%%%%%%%%%%%%

%As towers give rise to spectral sequences, so does the Goodwillie tower, and this property of an analytic functor can be read as a statement about convergence of the spectral sequence associated to the Goodwillie tower of $F$. 

Some of the earliest and most powerful results of Goodwillie calculus relate to analyticity and other properties which follow. Examples of $1$-analytic functors include $\I_{\Top}$,Waldhausen's algebraic $K$-theory functor, and $TC$, the topological cyclic homology of a space. For a $\rho$-connected CW complex $K$, the functor $X \mapsto \Omega^\infty \Sigma^\infty \Map (K, X)$ is $\rho$-analytic. For $Q:= \Omega^\infty \Sigma^\infty$, for each $i$, $Q^i$ is 0-analytic and $\Z_\infty X\simeq \holim_\Delta Q^i X$ is a 0-analytic functor, as holims of analytic functors are analytic.

%%%%%%%%%%%%%%%%%%%%%%%%%%%%%%%%%%%%%%%%
% Monads
\subsection{Monads M and left/right M-Functors}\label{sec:bkMod} 
%%%%%%%%%%%%%%%%%%%%%%%%%%%%%%%%%%%%%%%%
Monads are also sometimes called ``triples", especially in the more algebraic literature, and in some of the Goodwillie calculus constructions such as those of \cite{JM-cotriples,BEJM14, BJM14}.

We first recall relevant definitions of a monad $M$ and $M$-Functor, which is the functor extension of the notion of a module over the monad $M$:

%%%%%%%%%%%%%%%%
\begin{defn}\cite[p.133]{maclane}
A monad $M=f
 M,\eta,\mu\rangle$ in a category $C$ consists of a functor $M: C \ra C$ and  two natural transformations 
\[
\eta: \text{Id}_C \ra M \hspace{1 in} \mu : M^2 \ra M
\]
which make the following commute
\[
\xymatrix{
M^3 \ar[r]^{M \mu} \ar[d]^{\mu M} & M^2\ar[d]^{\mu}\\
M^2 \ar[r]^{\mu} & M
}
\hspace{.5 in}
\xymatrix{
\text{Id}_C\circ  M \ar@{=}[d] \ar[r]^{\eta M} & M^2 \ar[d]^{\mu}& \ar[l]_{M \eta} M \circ \text{Id}_C \ar@{=}[d]\\
M \ar@{=}[r]& M \ar@{=}[r]& M\\
}
\]

\end{defn}
%%%%%%%%%%%%%%%%

%%%%%%%%%%%%%%%%
\begin{defn}\cite[p.136]{maclane}
If $M=\langle M,\eta,\mu\rangle$ is a monad in a category $C$, we have notions of left and right $M$-module as follows.
A left $M$-module  (referred to in \cite[p.136]{maclane} as an $M$-algebra) $\langle x,h\rangle$ is a pair consisting of an object $x\in C$ and an arrow $h:Mx \ra x$ of $C$ which makes both of the following diagrams commute (associativity law, unit law):
\[
\xymatrix{
M^2 x \ar[r]^{Mh} \ar[d]_{\mu_x} & Mx\ar[d]_h & x \ar[r]^{\eta_x}\ar[dr]^{1}& Mx\ar[d]^h\\
Mx \ar[r]^h &x &&x\\
}
\]
A right $M$-module $\langle x',h'\rangle$ is a pair consisting of an object $x'\in C$  and an arrow $h':xM \ra x$ of $C$ (here $xM$ means $x$ with a right $M$-action) which makes both of the following diagrams commute (co-associativity law, co-unit law):
\[
\xymatrix{
x' M^2  \ar[r]^{h'M} \ar[d]_{\mu_{x'}^\join} & x' M\ar[d]_{h'} & x' \ar[r]^{\eta_{x'}^\join}\ar[dr]^{1}& x'M\ar[d]^{h'}\\
x' M \ar[r]^{h'} &x' &&x'\\
}
\]
\end{defn}
%%%%%%%%%%%%%%%%

The following is a slight modification of Definition 9.4 from \cite{May}. There it is called an $M$-functor. We re-name it so as to be able to talk about functors with both right and left $M$ actions. This is the functor-level analog of being an $M$-module. 

%%%%%%%%%%%%%%%%
% M-Functor
\begin{defn}
Let $(M, \mu, \eta)$ be a monad in $C$. A right $M$-functor $(G, \lambda)$ in a category $D$ is a functor $G: C \ra D$ together with a natural transformation of functors $\lambda :GM \ra G$ such that the following diagrams are commutative
\[
\xymatrix{
G M^2 \ar[r]^{G \mu} \ar[d]_{\lambda} & GM \ar[d]^{\lambda} \\
GM \ar[r] ^{\lambda} & G .\\
}
\hspace{.25 in}
\text{ and }
\hspace{.25 in}
\xymatrix{
G \ar[r]^{G \eta}\ar@{=}[dr] & GM\ar[d]^{\lambda}\\
 & G\\
}
\]
A left $M$-functor $(G, \lambda')$ in a category $D$ is a functor $G: C \ra D$ together with a natural transformation of functors $\lambda' :MG \ra G$ such that the following diagrams are commutative
\[
\xymatrix{
M^2 G\ar[r]^{\mu G} \ar[d]_{\lambda'} & M G \ar[d]^{\lambda'} \\
M G \ar[r] ^{\lambda'} & G .\\
}
\hspace{.25 in}
\text{ and }
\hspace{.25 in}
\xymatrix{
G \ar[r]^{\eta G}\ar@{=}[dr] & MG\ar[d]^{\lambda'}\\
 & G\\
}
\]
\end{defn}
%%%%%%%%%%%%%%%%

Note that for an adjoint pair of functors $L:A \ra B$, $R:B \ra A$, with unit and counit $\eta: Id_A \ra RL, \epsilon: LR \ra Id_B$ we have a natural monad $M:= RL$ with multiplication $\mu:= R \epsilon L: RLRL \ra RL$ and unit given by the unit of the adjunction.

%%%%%%%%%%%%%%%%
%%%%%%%%%%%%%%%%
\begin{prop}\label{prop:RFL}
For an adjoint pair of functors $L:A \ra B$, $R:B \ra A$ and an endofunctor $F: B \ra B$, we have that $RFL$ has the structure of a left and right $RL$-functor, and dually $LFR$ left and right $LR$-functor. 
\end{prop}
%%%%%%%%%%%%%%%%
%%%%%%%%%%%%%%%%

We will illustrate this for the bimodule over the monad case and the dual proof follows by dualizing our arguments and flipping the arrows in our diagrams. The structure map $\lambda := R \epsilon FL$ and costructure map $\lambda':= R F \epsilon L$. 

\if false
Left $RL$-functor structure:
\[
\xymatrix{
RL\circ RL\circ RFL \ar[rr]^{RLR \epsilon FL} \ar[d]_{R\epsilon L RFL} & &RL\circ RFL\ar[d]_{R \epsilon FL} & RFL \ar[rr]^{\eta RFL}\ar@{=>}[drr]& & RL\circ RFL\ar[d]^{R \epsilon FL}\\
RL\circ RFL\ar[rr]^{R \epsilon FL} && RFL &&& RFL\\
}
\]

Right $RL$-functor structure:
\[
\xymatrix{
RFL \circ  RL\circ RL  \ar[rr]^{R F \epsilon L RL} \ar[d]_{ RFL R\epsilon L} & &RFL \circ  RL\ar[d]_{R F \epsilon L} & RFL \ar[rr]^{RFL \eta}\ar@{=>}[drr]&& RFL \circ  RL\ar[d]^{R F \epsilon L}\\
RFL \circ  RL \ar[rr]^{R F \epsilon L} &&RFL &&&RFL\\
}
\]

{\color{blue}
\fi 

Right $RL$-functor structure:
\[
\xymatrix{
RFL \circ  RL\circ RL  \ar[rr]^{\lambda } \ar[d]_{ (RFL) \mu } & &RFL \circ  RL\ar[d]_{\lambda } & RFL \ar[rr]^{(RFL) \eta}\ar@{=>}[drr]&& RFL \circ  RL\ar[d]^{\lambda}\\
RFL \circ  RL \ar[rr]^{\lambda } &&RFL &&&RFL\\
}
\]

Left $RL$-functor structure:
\[
\xymatrix{
RL\circ RL\circ RFL \ar[rr]^{\lambda'} \ar[d]_{\mu (RFL)} & &RL\circ RFL\ar[d]_{\lambda'} & RFL \ar[rr]^{\eta (RFL)}\ar@{=>}[drr]& & RL\circ RFL\ar[d]^{\lambda'}\\
RL\circ RFL\ar[rr]^{\lambda'} && RFL &&& RFL\\
}
\]

%}

%%%%%%%%%%%%%%%%%%%%%%%%%%%%%%%%%%%%%%%%
% Quillen Functors
%%%%%%%%%%%%%%%%%%%%%%%%%%%%%%%%%%%%%%%%
\subsection{Quillen Functors} \label{sec:QFunc}
%%%%%%%%%%%%%%%%%%%%%%%%%%%%%%%%%%%%%%%%
We review the definitions of Quillen functor and adjunction and include a list of useful properties. 

\begin{defn}\cite[14.1.]{Dwyerhirschetc}\label{defn:Qadj}
Given two model categories M and N, a Quillen adjunction is an adjunction
\[
\xymatrix{f : M \ar@<2 pt>[r] & N : f'\ar@<2 pt>[l]}
\]
of which
\begin{itemize}
\item[(i)]  the left adjoint, $f$, is a left Quillen functor; a functor which preserves cofibrations and trivial cofibrations, and
\item[(ii)] the right adjoint, $f'$, is a right Quillen functor; a functor which preserves fibrations and trivial fibrations.
\end{itemize}
\end{defn}

\noindent We also provide a list of properties which these functors satisfy.
\begin{prop}\cite[14.2.(ii)-(iii)]{Dwyerhirschetc} \label{prop:opQ}
Quillen functors satisfy the following properties:
\begin{itemize}
\item[(i)] Every right adjoint of a left Quillen functor is a right Quillen functor and every left adjoint of a right Quillen functor is a left Quillen functor. 
\item[(ii)] The opposite of a left Quillen functor is a right Quillen functor and the
opposite of a right Quillen functor is a left Quillen functor. 
\end{itemize}
\end{prop}

\noindent We are considering adjoint pairs $(L_n, R_n)$ and $(L^n, R^n)$ where $L^n := R_n^{op}$. Thus, using this proposition, we need to only show one of the four functors is Quillen.

%%%%%%%%%%%%%%%%%%%%%%%%%%%%%%%%%%%%%%
%%%%%%%%%%%%%%%%%%%%%%%%%%%%%%%%%%%%%%
%
%  section
\section{Proofs of main results} \label{sec:proofs}
%
%
%%%%%%%%%%%%%%%%%%%%%%%%%%%%%%%%%%%%%%
%%%%%%%%%%%%%%%%%%%%%%%%%%%%%%%%%%%%%%

%\input{proof-main-results}
\noindent We will prove in this section the following theorem and its consequences: %\\

%%%%%%%%
\begin{thm}\label{thm:main}
Let $F$ be a reduced homotopy endofunctor of topological spaces. For all $n\geq1$, there are adjoint functors  $R_n, L_n$  such that $\T_n F$ is weakly equivalent to $R_n F L_n$. For $F$ strongly reduced, we have that $\T_n F=R_n F L_n$. In particular, $\T_n \I = R_nL_n$, i.e. has the structure of a monad for all $n$. 
\end{thm}
%%%%%%%%
%
%\noindent \textbf{Theorem \ref{thm:main}}\; 
%\textit{Let $F$ be a reduced homotopy endofunctor of topological spaces. For all $n\geq1$, there are adjoint functors  $R_n, L_n$  such that $\T_n F$ is weakly equivalent to $R_n F L_n$. For $F$ \emph{strongly} reduced, we have that $\T_n F=R_n F L_n$. In particular, $\T_n \I = R_nL_n$, i.e. has the structure of a monad for all $n$. }\\

The relevant diagram of adjunctions, with left adjoints on the top, is as follows:
\[
L_n:
\xymatrix{
\Top \ar@<+5 pt>[rr]^-{S\mapsto -\x \Delta^S }&& \ar[ll]^-{\holim} \Fun (\mathscr{P}_0([n]), \Top) \ar@<+5 pt>[r]^{\text{red}} & \ar[l]^{\text{inc}}\widetilde{\Fun} (\mathscr{P}_0([n]), \Top). \\
}
:R_n
\]

%$L_n := \text{red} \circ (S\mapsto - \x \Delta^S )$ and $R_n := \holim \circ \text{ inc}$. 

%\noindent Here, we use the term \textbf{reduced} as Goodwillie did in \cite{GC3}, to mean that $F$ takes contractible spaces to contractible spaces. 

Key for this proof is to determine the correct categories to be working between.  %We note here that this is some of the hindrance in providing a similar proof for the $\P_n \I$'s -- the need to determine the correct intermediate categories.

%%%%%%%%%%%%%%%%%%%%%%%%%%%%%%%%%%%%%%
%%%%%%%%%%%%%%%%%%%%%%%%%%%%%%%%%%%%%%

%\subsection{}

We will first provide the proof in the case $n=1$ to motivate the general proof. 
Recall the definition of $\T_n F(X):= \underset{U \in \mathscr{P}_0([n])}{\holim} F(U \join X)$. For $n=1$, this yields the following homotopy pullback square: 

\[
\xymatrix{
\T_1 F(X) \ar[r]\ar[d]& F(CX)\ar[d]\\
F(CX)\ar[r] &F(\Sigma X).\\
}
\]

For reduced functors from based spaces to based spaces, $\T_1 F(X) \simeq \Omega F (\Sigma X)$, and $\Sigma, \Omega$ are adjoints between those categories. If we relax to reduced functors of unbased spaces, we have to be slightly more careful. % to get our adjunctions. 

There is a (clear) equivalence of categories between spaces and diagrams of the form 
\[
\ast \la X \ra \ast
\]
where $\ast$ is a point and $X$ is a space. However, the category of dual diagrams,
\[
\ast \ra Y \la \ast
\]
for $Y$ a space, is equivalent to the category of \textit{spaces with two basepoints}, which we denote by $\Top_{\ast_1 \; \ast_2}$.  We can see that we have an adjunction 
\[
\xymatrix{
\text{Unreduced Suspension}: \Top \ar@<5 pt>[r] & \ar[l] \Top_{\ast_1 \; \ast_2}: \text{Paths (between $ \ast_1$ and  $\ast_2$) }
}
\]
such that for $F$ reduced, $\T_1 F(X)$ is equivalent to $X \mapsto SX$ followed by $F$ (which remains in $\Top_{\ast_1 \; \ast_2}$ because $F$ is reduced) and then by taking paths. 

The general case will not involve spaces with a multitude of basepoints, but cubical diagrams $\X$ which are similarly ``reduced'', i.e.  $\X (S)$ is a point whenever $|S|=1$.

%%%%%%%%%%%%%%%%%%%%%%%%%%%%%%%%%%%%%%
\subsection{Proof of the general case, arbitrary n} \label{sec:genproof}
%%%%%%%%%%%%%%%%%%%%%%%%%%%%%%%%%%%%%%

We will be working with the categories of spaces, Top, of punctured diagrams of spaces,  Fun$(\mathscr{P}_0 ([n]), \Top) = \Top^{\mathscr{P}_0 ([n])}$ and of reduced punctured cubical diagrams of spaces, $\widetilde{\mathrm{Fun}}(\mathscr{P}_0([n]), \Top)$. Each diagram $\X \in \widetilde{\mathrm{Fun}}(\mathscr{P}_0([n]), \Top)$ has the property that $\X(S)$ is a point, for $|S|=1$. 

We now establish the adjunctions given below the statement of Theorem \ref{thm:main}, where $L_n := \text{red} \circ (S\mapsto - \x \Delta^S )$ and $R_n := \holim \circ \text{ inc}$.

%%%%%%%%%%%%%%%%%%%%%%%%%%%%%%%%%%%%%%%
%%%%%%%%%%%%%%%%%%%%%%%%%%%%%%%%%%%%%%%
%
\subsection{Holim and \texorpdfstring{$-\x \Delta^S$}{its adjoint}} 
%
%%%%%%%%%%%%%%%%%%%%%%%%%%%%%%%%%%%%%%%
%%%%%%%%%%%%%%%%%%%%%%%%%%%%%%%%%%%%%%%

The model which we use for the homotopy limit of a punctured cube $\X$ is 
 $\hom_{\Top^{\mathscr{P}_0([n])}}(\Delta^{-},\X (-))$. It has a natural left adjoint, which takes a space $X$ and sends it to the punctured cubical diagram $S \mapsto X \x \Delta^S$, $S\in \mathscr{P}_0([n])$:  
\[  
\hom_{\Top}(Y, \hom_{\Top^{\mathscr{P}_0([n])}}(\Delta^{-},\X (-))) 
\cong 
\hom_{\Top^{\mathscr{P}_0([n])}}(Y\x \Delta^{-}, \X (-))) .
\]
Note that  $Y\x(\Delta^S|_{S\in \mathscr{P}_0([n])})$ is precisely $S \mapsto Y \x \Delta^S$, $S\in \mathscr{P}_0([n])$. This adjunction is established in \cite[CH XI, \S 3]{BK}. This arises from the pointwise adjunctions,  which for each $S\in \mathscr{P}_0([n])$ are of the form 
\[  
\hom_{\Top}(Y, \hom_{\Top}\Delta^S,\X(S)) 
\cong 
\hom_{\Top}(Y\x \Delta^S, \X (S)) .
\]

%%%%%%%%%%%%%%%%%%%%%%%%%%%%%%%%%%%%%%%
%%%%%%%%%%%%%%%%%%%%%%%%%%%%%%%%%%%%%%%
%
\subsection{Reduction and inclusion} 
%
%%%%%%%%%%%%%%%%%%%%%%%%%%%%%%%%%%%%%%%
%%%%%%%%%%%%%%%%%%%%%%%%%%%%%%%%%%%%%%%

There is also a natural left adjoint to the inclusion of reduced punctured cubical diagrams into punctured cubical diagrams. This takes a diagram $\ycal$ to a diagram $\red (\ycal)$ such that 
\[
\red (\ycal)(S) := \mathrm{colim}(\ycal(S) \la \coprod \limits_{j \in S} \ycal(\{j\}) \ra S).
\]
It is not alarming that this is not a priori a \textit{homotopy} colimit, because we want a diagram that is honestly reduced, with $\red (\ycal)(S)$ to be a point, not just contractible, for $|S|=1$.

To establish this adjunction, it suffices to show that any map from $\ycal \in  \Top^{\mathscr{P}_0 ([n])}$ to $\X \in \widetilde{\mathrm{Fun}}(\mathscr{P}_0([n]), \Top)$ must factor through red$(\ycal):= S \mapsto \red (\ycal)(S)$.

Consider $n=2$, a map of punctured squares. 
\[
\scalebox{.8}{$
\xymatrix{
             & \ycal(0)\ar[dd] \ar[dr]\\ 
             &                    & \X (0)=\ast \ar[dd]\\ 
\ycal(1) \ar[r] \ar[dr] & \ycal(\{0,1\}) \ar[dr]\\
                    & \ast = \X (1) \ar[r] & \X (\{0,1\}) 
}$}
\]
This map of diagrams is the same as having a square of the following form: 
\[
\xymatrix{
\ycal(0) \smcoprod \ycal(1) \ar[d] \ar[r] & \ast \smcoprod \ast \ar[d]\\
\ycal(\{0,1\}) \ar[r] & X(\{ 0,1\}) \\
}
\]
The maps to the final space must factor through the colim of the rest, which is exactly $\red (\ycal)(\{0,1\})$. 

\[
\xymatrix{
\ycal(0) \smcoprod \ycal(1) \ar[d] \ar[r] & \ast \smcoprod \ast \ar[d]\ar@/^/[ddr]\\
\ycal(\{0,1\}) \ar@/_/[drr]\ar[r] & \red (\ycal)(\{ 0,1\})\ar[dr] \\
                  &                              & \X (\{0,1\})\\
}
\]
What remains is to explain how this extends to higher dimensional cubical diagrams. \\

There is a natural map from $\ycal$ to red$(\ycal)$: % One way to see this is to view $\red (Y)$, the reduction, as the cofiber of several maps at once into $Y(S)$, and 
%it is clear that for each 
Consider the collection of maps for each $j \in S$, $\ycal(j) \ra \ycal(S)$ instead as one map from the coproduct $\coprod_{j\in S}\ycal(j) \ra \ycal(S)$  and it is clear that $\ycal(S)$ maps to the cofiber of this map. 

Given that, and given a map of cubes $f:\ycal \ra \X$ for $\X$ reduced, we have 
\[
\xymatrix{
\ycal \ar[r]^{\text{f}} \ar[d]_{\red } & \X \ar[d]^{\red }_{\rcong}\\
\text{red}(\ycal) \ar[r]^{\red (\text{f})} & \X\\
}
\]
That is, $f$ must factor through the reduced cube. $\qed$

%%%%%%%%%%%%%%%%%%%%%%%%%%%%%%%%%%%%%%%
%%%%%%%%%%%%%%%%%%%%%%%%%%%%%%%%%%%%%%%
%
\subsection{Composed adjunctions and the topological join} 
%
%%%%%%%%%%%%%%%%%%%%%%%%%%%%%%%%%%%%%%%
%%%%%%%%%%%%%%%%%%%%%%%%%%%%%%%%%%%%%%%

Then $L_n = \text{red} \circ (S\mapsto - \x \Delta^S )$, the composition of the two left adjoints, sends a space $X$ to the diagram
\[
S \mapsto \colim (X \x \Delta^S \la X \x S \ra S)
\]

Note that a model for the join of two spaces, $X$ and $Y$, is the following, where C is the cone: 
\[
 \colim (X \x CY \la X \x Y \ra Y)
\]
Since the map  $X\x CY \la X\x Y$ is a cofibration, this colim is also a model for the homotopy colim of the same diagram with $X$ instead of $X \x CY$. 

For a set $S$ considered as a discrete space, the natural inclusion $CS  \ra \Delta^S$ is a cofibration and a homotopy equivalence and makes the next diagram commute. 
%.  a natural map one way which makes the diagram on the following page commute and which is a homotopy equivalence.
That is, we have the following map of diagrams with all vertical arrows homotopy equivalences. 
Moreover, the top left horizontal map is a cofibration and this is enough to ensure the induced map of colims of the rows is an equivalence. %
% and at least one is also a cofibration, so they have homotopic colimits. 
This suffices because, thanks to \cite[Appendix]{dugger-appendix}, we do not need cofibrancy on the objects as well.  

% formerly Figure 3
%\begin{figure}{h}
\[
\xymatrix{
X \x CS \ar@{{)}->}[d]  & X \x S\ar[d]^{id_{X\x S}} \ar[r]\ar[l]& S\ar[d]^{id_S}\\
X \x \Delta^S                  & X \x S                  \ar[r]\ar[l]& S\\
}
\]
%\caption{Maps of join models}
%\label{fig:joins}
%\end{figure}

\noindent As a result, $L_n$ is not just a colim but a hocolim, and its pushout it is homotopic to $X \join S$, so we may take it as a model for the join.  That is, $L_n (X) = S \mapsto X \join S$. 

%%%%%%%%%%%%%%%%%%%%%%%%%%%%%%%%%%%%%%%
%%%%%%%%%%%%%%%%%%%%%%%%%%%%%%%%%%%%%%%
%
\subsection{Relating this adjunction to the Goodwillie Calculus} 
%
%%%%%%%%%%%%%%%%%%%%%%%%%%%%%%%%%%%%%%%
%%%%%%%%%%%%%%%%%%%%%%%%%%%%%%%%%%%%%%%

Recall that $\T_n F(X)$ is formed by first applying $F$ to the diagram $S \mapsto X \join S$, for $S\in \mathscr{P}_0 ([n])$ and taking the homotopy limit. That is, it may be written as $\holim F \circ L_n$.

For a very general homotopy functor $F$, it will not be true that $F$ of a reduced punctured cube will again be a reduced punctured cube. The condition necessary is that $F$ is \textit{strongly reduced}. 

As we will want to compare our results with those of Biedermann and Dwyer, it is important to note that they restrict to functors which are of this type, specifically, with spaces replaced by based, connected simplicial sets. 

Given such an $F$, and that the holim of punctured cubes is the same as $\holim\circ$ inc $=R_n$, we see that $\T_n F = R_n F L_n$.

That is, we have established an adjunction between Top and $\widetilde{\Fun}(\mathscr{P}_0([n]), \Top)$ for each $n$ such that $\T_n F(X) = R_n F ( L_n X)$. %
%%%%%%%%%%%%%%%%%%%%%%%%%%%%%%%%%%%%%%%
\subsubsection{Functors which are reduced, not strongly reduced} 
%%%%%%%%%%%%%%%%%%%%%%%%%%%%%%%%%%%%%%%
In terms of other applications the author has in mind, it would be best if we could be less restrictive with our functors. %restrict ourselves not just to functors which are based. What can we then do to rectify the situation if $F$ is \textit{almost} good enough, if $F$, \textit{up to homotopy}, takes a point to a point? That is, if $F$ is \textit{reduced}. 

For $F$ reduced, consider now $\T_n \tilde{F} := R_n \circ \red \circ F\circ L_n$.  \\

If $F$ takes values in path-connected spaces, then we may contract $F(S)$, $|S|=1$ to a point for all $S$ and not disturb the homotopy type of the homotopy limit. That is, for $F$ reduced and taking values in connected spaces, $\T_n F \overset{\simeq}{\ra} \T_n \tilde{F}$ is a 
weak homotopy equivalence.

\[
\xymatrix{
F(CX) \ar@{->>}[d]^{\simeq} \ar[r]  & F(\Sigma X) \ar@{=}[d] &\ar[l] F(CX)\ar@{->>}[d]^{\simeq}\\
\ast \ar[r]  & F(\Sigma X) &\ar[l] \ast \\
}
\]
It is important to note that if $F$ is strongly reduced,  then $\T_n F \overset{\cong}{\ra} \T_n \tilde{F}$ is a point-wise homeomorphism.

%%%%%%%%%%%%%%%%%%%%%%%%%%%%%%%%%%%%%%
\subsection{Quillen adjunction}
%%%%%%%%%%%%%%%%%%%%%%%%%%%%%%%%%%%%%%
By Definition \ref{defn:Qadj}, to establish that our adjunctions are Quillen pairs, since we already have that they are adjunctions, we just need to check that either $L_n$ preserves cofibrations and trivial cofibrations or $R_n$ preserves fibrations and trivial fibrations. 

Top has the usual model category structure with fibrations the Serre fibrations;  cofibrations the retracts of relative cell complexes; and weak equivalences the weak homotopy equivalences. Both diagram categories will be taken with the levelwise model structure induced by this model structure in Top. We have that
\[
\begin{array}{lcr}
X \mapsto (S \mapsto X \join S) &\hspace{1cm} &S \in \mathscr{P}([n])
\end{array}
\]
preserves cofibrations and trivial cofibrations, as follows. Given a (trivial) cofibration $X\cofra Y$, consider the cube

\[
\xymatrix{
X \x S \ar[rr] \ar[dd]\ar[dr] & &S\ar[dd] \ar[dr]\\
& X \x \Delta^S \ar[rr] \ar[dd]& & X \join S\ar[dd]\\
Y \x S \ar[rr] \ar[dr] & &S\ar[dr]\\
& Y \x \Delta^S \ar[rr] & & Y \join S\\
}
\]

\if false % OLD CUBE, WRONG FACES 
\[
\xymatrix{
X \x \Delta^S \ar[rr] \ar[dd]\ar[dr] & &S\ar[dd] \ar[dr]\\
& X \ar[rr] \ar[dd]& & X \join S\ar[dd]\\
Y \x \Delta^S \ar[rr] \ar[dr] & &S\ar[dr]\\
& Y \ar[rr] & & Y \join S\\
}
\]
\fi

\noindent This cube is cocartesian because the top and bottom squares (the ones with pushout $X\join S$ and $Y \join S$, respectively) are cocartesian. Moreover, and the three vertical maps $X\ra Y$, $X \x \Delta^s \ra Y\x \Delta^s$, and $S \ra S$ are all (trivial) cofibrations. 

As cofibrations in Top are stable under cobase change, the map $X \join S \ra Y \join S$ is also a cofibration.

In the case of considering a trivial cofibration, homotopy invariance of homotopy colimits (i.e. topological join) yields that the map $X \join S \ra Y \join S$ is a weak equivalence, i.e. it is also a trivial cofibration. This will hold for all $S$ and (trivial) cofibrations are defined levelwise, so we have shown that this is not just a left adjoint but also a left Quillen functor.

%[Obviously] if we use the levelwise model structure, this is immediate. 

\begin{rem}
We would like to point out that if we start with fibrations of cosimplicial spaces and consider the induced cubical diagrams, these will still be fibrations in the levelwise structure as Reedy fibrations are also levelwise fibrations.  These diagrams are obtained by precomposing with the functor $\circ c_n: \mathscr{P}_0([n])\ra \Delta_{\leq n}$ which sends $S$ to $[\# S-1]$ and inclusions to the induced coface maps. So, fibrations of cosimplicial spaces are sent to fibrations in Top when following $\circ c_n$ by this Quillen adjunction.  
\end{rem}

Using Prop \ref{prop:opQ}(iii), we can conclude that the duals will also be Quillen adjoints since our model for hocolim is holim$^{op}$.

%%%%%%%%%%%%%%%%%%%%%%%%%%%%%%%%%%%%%%
%%%%%%%%%%%%%%%%%%%%%%%%%%%%%%%%%%%%%%
%
%  section
\section{LS cocategory and related corollaries} \label{sec:lscocat}
%
%g
%%%%%%%%%%%%%%%%%%%%%%%%%%%%%%%%%%%%%%
%%%%%%%%%%%%%%%%%%%%%%%%%%%%%%%%%%%%%%

%\input{sec-LScocat}
The purpose of this section is expand on the relationship between LScocategory and the constructions of Goodwillie calculus, including proofs and more details around the corollaries of Theorem \ref{thm:main} and also of Proposition \ref{prop:radconv}.

%\subsection{Hopkins' definition of symm LScocat and the $\T_n$}\label{sec:hopkinsLS}

%\begin{rem} 
In the introduction of \cite{me-AGT} , it is explained exactly how to translate Hopkins' definition of symmetric LScocategory\cite[Section 3, p221-222]{hopkins} and how to translate from his language to ours.  The important part to note is that he defines, for a given space $X$, a functor $F^n$, as the homotopy inverse limit of a (punctured) cube; this is exactly $\T_n \mathbb{I} (X)$. 
%\end{rem}

%We give here an explanation of Hopkins' definition of symmetric LScocategory\cite[Section 3, p221-222]{hopkins} and how to translate from his language to ours.   He lets $C_n$ be 
%what we call $\mathscr{P}_0 ([n])$. He defines, for a given space $X$, a functor $F^n$, as the homotopy inverse limit of a (punctured) cube. For $A \in \mathscr{P}_0 ([n])=:C_n$,  the $A$-indexed position of this $(n+1)$-cube is the homotopy colimit of X mapping to $|A|$ different copies of a point, which we will explain shortly. He denotes this by $F^n A$. Regarding $A$ as a finite ordered set, we can view $F^nA$ as the homotopy pushout of the following: 
%
%\[
%\scalebox{.85}{$
%\xymatrix{
%                 &X\ar[dl]\ar[d]\ar[drr]\\
%\{0\} & \{1\} & \cdots & \{|A|-1\} \\
%} 
%$}
%\]
%
%\noindent We replace these maps by cofibrations (since we are taking a homotopy colimit), giving us that we are pushing out over the following diagram:
%\[
%\scalebox{.85}{$
%\xymatrix{
%                 &X\ar[dl]\ar[d]\ar[drr]\\
%\{0\}\join X & \{1\}\join X & \cdots & \{|A|-1\}\join X \\
%} 
%$}
%\]
%
%\noindent That is, the $A$-indexed position of this $(n+1)$ cube is  $A\join X$. 
%
%Then $F^n := \holim_{A \in  \mathscr{P}_0 ([n])} F^n A \sim\holim_{U \mathscr{P}_0 ([n])} U \join X$. That is, we have shown that his $F^n$'s exactly the $\T_n \I(X)$'s, i.e. we have established:\\

Consequently, 

\begin{repprop}{def:LScocatTn}
A space $X$ has symmetric-LScocategory less than or equal to $n$ if and only if $X$ is a homotopy retract of $\T_n \I (X)$.
\end{repprop}
%%%%%%%%

%\noindent He also constructs a tower of these $F^n$'s: 
% 
%\[
%(\cdots \ra \holim F^n \ra \holim F^{n-1} \ra \cdots \holim F^1)
%\]
%which is therefore our $\T_n \I$ tower, 
% \[
% (\cdots \ra \T_n\I(X) \ra \T_{n-1}\I(X) \ra \cdots \T_1 \I(X)).
% \]

%%%%%%%%%%%%%%%%%%%%%%%%%%%%%%%%%%%%%%
%
\subsection{Results and further proofs} \label{sec:lscocatproofs}
%
%%%%%%%%%%%%%%%%%%%%%%%%%%%%%%%%%%%%%%
%%%%%%%%
This section contains the statements and proofs of the corollaries of Theorem \ref{thm:main} and Corollary \ref{cor:TnI} which relate to the relationship with LS cocat. \\

%%%%%%%

%The rigorous definition of being left or right $M$-functors (for $M$ a monad) may be found in section. 

The following corollary follows immediately from combining Proposition \ref{prop:RFL} with Theorem \ref{thm:main}.

%%%%%%%%
\begin{cor} \label{cor:TnI}
The functors $\T_nF$ are left and right $\T_n \I$-functors, as are the $\P_nF$.   
\end{cor}
%%%%%%%%
%
%
%\noindent \textbf{Corollary \ref{cor:TnI}}:
%\textit{$\T_nF$ are left and right $\T_n \I$-functors, as are the $\P_nF$.}\\
%%%%%%%%%

%\begin{proof}
%\end{proof}

Corollary \ref{cor:TnI} establishes structure that provides maps which express $\T_n F(X)$ as a retract of $\T_n\I(\T_n F(X))$, which is exactly what one needs to say  that a space has LS cocat $\leq n$, by Proposition \ref{def:LScocatTn}. So, combining with Theorem \ref{thm:main} we also have

%%%%%%%%
%\noindent \textbf{Corollary \ref{cor:sym}}: 
%\textit{$\T_nF$ takes values in spaces of symmetric LS cocat $\leq n$, as do the $\P_nF$.}\\
%%%%%%%%%
\begin{cor}\label{cor:sym}
For each $n \geq 1$, the functor $\T_nF$ takes values in spaces of symmetric LS cocat $\leq n$, as do the $\P_nF$.   
\end{cor}
%%%%%%%%

\begin{proof}
Once we have the result for each $\T_nF$ and $\T_n^k F$, these homotopy retract maps clearly induce the same structure on the homotopy colimit of the $T_n^kF$'s, $\P_n F$. %:= \hocolim_k \T^k_n F$.

This follows immediately from the retract structure established in Corollary \ref{cor:TnI}. There is always a map, for $X$ a space, $X=\I(X) \ra \T_n \I (X)$.   So also for a space $\T_n F(X)$, we have a map $\T_n F(X) \ra \T_n \I \circ \T_n F(X)$. Writing in the adjunctions, we have the following: 
\[
\begin{array}{ccc}
\T_n FX  &\lra & \T_n \I \circ \T_n F(X)\\
\rcong & & \rcong\\
R_n F L_n & \lra   & R_n L_n R_n F L_nX\\
\end{array}
\]

The counit of the adjunction, $\epsilon: L_n R_n \I \ra \I$, provides our map 
\[
R_n L_n R_n F L_nX \ra R_n F L_n
\]
Recall that the counit and the unit of the adjunction give that the following composition
\[
\xymatrix{
%L_n \ar[r]^{L_n \eta}& L_n R_n L_n \ar[r]^{\epsilon L_n} & L_n\\
R_n \ar[r]^{\eta R_n}& R_n L_n R_n \ar[r]^{R_n \epsilon} & R_n\\
}
\]
is the identity. 

We can apply this to $FL_n X$ to realize it as our maps 
\[
%\xymatrix{
R_n F L_nX \ra R_n L_n R_n F L_nX \ra R_n F L_n
%}
\]
%%%%%%%%%%%%%%%%%%%
\end{proof}

Theorem 2.1 of \cite{chorny-scherer} states that the Whitehead products of length $\geq n+1$ vanish in $\P_n F(X)$ for any space $X$. Corollary \ref{cor:WL} recovers and extend this result to the $\T_n$ a well: \\

%%%%%%%%
\begin{cor}\label{cor:WL}
For every space $X$, the spaces $\T_nF(X)$ and $\P_n F(X)$ have Whitehead length $n$.
\end{cor}
%%%%%%%%
%%%%%%%%
%\noindent \textbf{Conjecture\ref{cor:WL}}:
%\textit{The Whitehead products of length $\geq n+1$ vanish for $\T_n F(X)$ and $\P_n F(X)$ for any space $X$.}\\
%%%%%%%%

\begin{proof}	
Combining Corollary \ref{cor:sym} with the following chain of inequalities due to \cite{BersteinGanea, hopkinsthesis, GaneaLS}, we conclude our result: 

%\[
\begin{equation}\label{ineq-nil}
\text{Whitehead length}(X) \leq  \text{nil}(\Omega X) \leq \text{ind LScocat}(X) \leq \text{sym LScocat}(X).
\end{equation}
%\]
\end{proof}
 
 %%%%%%%%
\begin{prop}\label{prop:radconv}
Let F be a $\rho$-analytic homotopy endofunctor of spaces for some $\rho\geq 0$. Assume that there is also an $n > 0$ such that $\T_nF$ is $n$-excisive. Then $F(X)$ is weakly equivalent to $\P_n F(X)$ for all $X$ of connectivity $\geq \rho$. 
 \end{prop}
 %%%%%%%%
 
%\noindent \textbf{Proposition \ref{prop:radconv}}
%\textit{If F is a $\rho$-analytic homotopy endofunctor of spaces for some $\rho$ and also for some $n$, $\T_nF$ is $n$-excisive, then $F(X)$ is equivalent to $\P_n F(X)$ for all $X$ of connectivity $\geq \rho$. % in the radius of convergence of F.
%}\\

\begin{proof}
Proposition \ref{prop:analyt} stated that if F is $\rho$-analytic, then for $X$ at least $\rho$-connected, $\P_\infty F(X)\simeq F(X)$.

Corollary 1.4  of \cite{me-AGT} states that for $F$ a $\rho$-analytic homotopy endofunctor of spaces and some space $X$, 
\[
\P_\infty F(X)\overset{\simeq}{\ra} \holim (\cdots \ra \T_n^k F(X) \ra \cdots \ra \T_2^kF(X) \ra \T_1^kF(X) )
\]
for all $k\geq \text{max}(\rho-1-\text{conn}(X)-1,0)$. % Note that the above is the $k$th intermediate tower, the $k$th row of Figure \ref{fig:towers}

We have assumed that $\T_n F$ is n-excisive, which implies that $\T_n F$ is equivalent to $\P_nF$ %, since $\T_n F \overset{\simeq}{\ra} \P_n \T_n F$
\footnote{This is established by Goodwillie on the bottom of p.661 of \cite{GC3}.}. %
$\T_n F$ $n$-excisive means that  $\T_n (\T_n F) \simeq \T_n F$, in particular,  $\T_n^k F$ is also equivalent to $\T_n F$ and $n$-excisive. A similar argument shows that excisiveness holds for the higher $\T_{n+1}F$ and $\T_{n+1}^kF$. %
So the holim of each row will be $\P_n F$; in particular, we know the holim of the rows at level $k\geq \rho$ (i.e. greater than or equal to the analyticity of $F$) will have this property, therefore $\P_\infty F \simeq \P_n F$.

And on its radius of convergence, $F$ is equivalent to $\P_\infty F$, which we just said was equivalent to $\P_n F$. So, on its radius of convergence--for $X$ with connectivity $\geq \rho$ -- $F$ is equivalent to $\P_nF$.
\end{proof}

%%%%%%%%%%%%%%%%%%%%%%%%%%%%%%%%%%%%%%
%%%%%%%%%%%%%%%%%%%%%%%%%%%%%%%%%%%%%%
%
%  section
\section{Dual adjunction} \label{sec:dualproofs}
%
%
%%%%%%%%%%%%%%%%%%%%%%%%%%%%%%%%%%%%%%
%%%%%%%%%%%%%%%%%%%%%%%%%%%%%%%%%%%%%%

%\input{proofs-duals}
Historically, dual calculus has only been defined rigorously in the cotriple calculus, see \cite{Dual}. To provide now a dual calculus for excisive functors, we choose here %When we wish to speak of a dual calculus for excisive functors, we have two choices. Either we define this rigorously, then make our theorems about their structure, or we 
to define the constructions using our adjunctions and then prove that they  give rise to the (homtopy-universal) $n$-co-excisive approximations to a functor. %We choose the latter route. 

We recall that the model for hocolim of a copunctured diagram we are using is equivalent to %-- an element of $\Fun (\mathscr{P}^1([n]), \Top)$-- which we use can also be viewed as 
$\holim^{op}$ for our holim model. %, since  $\Fun (\mathscr{P}_0([n]), \Top)^{op} = \Fun (\mathscr{P}^1([n]), \Top)$. 
Let $R^n:=R_n^{op}$, hocolim precomposed with an inclusion. We will describe its adjoint, which we call $L^n$. 

We will need the category  $\utilde{\mathrm{Fun}}(\mathscr{P}^1([n]), \Top)$, the subcategory of '``co-reduced'' co-punctured cubical diagrams of spaces. That is, each diagram $\X \in \utilde{\mathrm{Fun}}(\mathscr{P}^1([n]), \Top)$ has the property that $\X([n]-S)$ is contractible, for $|S|=1$.  Dual to \textit{red}, which appears in the adjunction in Theorem \ref{thm:main}, there is a co-reduction functor that takes a copunctured cube to a coreduced one. Then %\\

%%%%%%%%
\begin{prop} \label{prop:dualadj}
For each $n \geq 1$, there are adjunctions between these categories as follows:
\[
\begin{array}{c}
L^n:
\xymatrix{
\utilde{\Fun}(\mathscr{P}^1 ([n]), \Top) \ar@<+3 pt>[r]^-{ \text{inc}}& \ar@<+3pt>[l]^{ \text{cor}}  \Fun (\mathscr{P}^1([n]), \Top) \ar@<+3 pt>[rr]^-{\hocolim} & &\ar@<+3 pt>[ll]^-{X \mapsto ([n]-S \mapsto X^{\Delta^S})} \Top \\
}:R^n\\
\\
L^n:= \hocolim \circ \text{inc} \;\;\;\;\;\; R^n:= \text{cor} \circ (X \mapsto ([n]-S \mapsto X^{\Delta^S}))\\
\end{array}
\]
\end{prop}
%%%%%%%%
 %%%%%%%%%%%%%%%%%%%%%
%\noindent \textbf{Proposition \ref{prop:dualadj}}
%\textit{For each $n \geq 1$, there are adjunctions between these categories as follows:}
%\[
%L^n:
%\xymatrix{
%\utilde{\Fun}(\mathscr{P}^1 ([n]), \Top) \ar@<+3 pt>[r]^-{inc}& \ar@<+3pt>[l]^{cor}  \Fun (\mathscr{P}^1([n]), \Top) \ar@<+3 pt>[r]^-{\hocolim} & \ar@<+3 pt>[l]^-{X \mapsto ([n]-S \mapsto X^{\Delta^S})} \Top .\\
%}:R^n
%\]
%%%%%%%%%%%%%%%%%%%
%

%%%%%%%%%
%\noindent \textbf{Definition \ref{defn:dualadj}}
%\textit{Given our adjoint pair $R^n, L^n$, for all $n\geq1$, we define $\T^n F:= L^n FR^n$; $T^n \I$ is then the comonad $L^n R^n$. There is a natural map $t^nF: \T^n F \ra F$ which is the map from a (ho)colimit of a co-punctured diagram to its final entry. }\\
%%%%%%%%
\begin{defn} \label{defn:dualadj}
%\textit{
Given our adjoint pair $R^n, L^n$, for all $n\geq1$, we define $\T^n F:= L^n FR^n$; $T^n \I$ is then the comonad $L^n R^n$. There is a natural map $t^nF: \T^n F \ra F$ which is the map from a hocolimit of a co-punctured diagram to its final entry. 
\end{defn}
%%%%%%%%

The previous first step was to consider the natural model for a homotopy limit and its adjoint. In a similar way, we can consider the natural model for a homotopy colimit, following Dugger\cite[Section 8.10]{duggerPrimer}.  For more on tensoring diagrams, see section \ref{sec:bckgd}.

For any copunctured diagram $\X : \mathscr{P}^1 ([n]) \ra \Top$, we form its homotopy colimit by tensoring with the diagram $S \mapsto \Delta^S$ ; this is our $B(- \downarrow I)^{op}$ . Thanks to the fact that this is a tensor product, we get an adjunction very similar to the homotopy limit case. The right adjoint now sends $X$ in Top to the copunctured cube $[n]-S \mapsto X^{\Delta^S}$, which is the dual of sending $X$ to the punctured cube $S \mapsto X \x \Delta^S$. %

We next need the canonical co-reduced diagram. The first step is to dualize the process of taking a space and producing a diagram 
\[
S \mapsto \colim (X \x \Delta^S \la X \x S \ra S).
\]
%by dualizing the processes we composed to get this diagram. 

We are using the convention now of indexing our diagrams in $\mathscr{P}^1([n])$ by sets in $\mathscr{P}_0([n])$ by considering where we send $[n]-S$ for varying $S$ in $\mathscr{P}_0([n])$. Dualizing the construction for reduced diagram yields 
\[
 [n]-S \mapsto \lim \left( \begin{array}{c}\xymatrix{ & \prod_S\ast  \ar[d]\\ \X([n]-S) \ar[r]& \prod \limits_{s\in S} \X([n]-s)} \end{array}\right).
\]
%Now, about the maps. 
The map into the product is the map induced by each map from the original diagram between $X([n]-S)$ and $X([n]-s)$ for all $s \in S$. The map from $\ast =\prod \limits_{s \in S} \ast$ is the choice of a point in each copy of $\X([n]-s)$.  If each space is based, this is the base point. If not, we need to start with connected spaces and have made a choice of base point. This map is the opposite of collapsing each space in $\coprod \limits_{s \in S} \X(s)$ to the point indexing it. Pre-composing with the right adjoint to hocolim yields 
\[
R^n: X \mapsto \left( [n]-S \mapsto \lim \left( \begin{array}{c}\xymatrix{ & \ast \ar[d]\\ X^{\Delta^{S}} \ar[r]& \prod \limits_{s\in S} X} \end{array}\right)\right).
\]
We using that $X^{\Delta^0} \cong X$, so the final element is unchanged. %, % the following \noindent 
%which is our $R^n$. 
Note that this limit is a model for the homotopy limit of the inner diagram, where $X$ is replaced by $X^{\Delta^S}$ and the bottom map is a fibrant replacement of the diagonal. Then $L^n$ is inclusion followed by hocolim. 

For $S=[n]$, the above construction produces $X$.  For every singleton $s \in S$, we get $P_s X$, i.e. paths in $X$ based at whatever point was chosen by the map of $s$ into $X$. If $X$ is based already, these are all copies of the ``normal'' based path space, $P_\ast X$. Then we have at the $S=\{i,j\}$-indexed spoos loops on $X$.

For example, for $X$ based and $[n]=\{0,1\}$, $R^1 X$ is
\[
P_\ast X \la \Omega X \ra P_\ast X.
\]

%\appendix 
%%%%%%%%%%%%%%%%%%%%%%%%%%%%%%%%%%%%%%%%
% DUAL Goodwillie Calc
%%%%%%%%%%%%%%%%%%%%%%%%%%%%%%%%%%%%%%%%
\section{Dual Goodwillie Calculus}\label{sec:bkDGC} 
%%%%%%%%%%%%%%%%%%%%%%%%%%%%%%%%%%%%%%%%

%\input{appendix-dual}
%{\color{blue}
%} %These results can be seen as providing motivation for renewed investigation of this dual tower.

%In this appendix, 
In this section, we further develop the dual calculus theory, which includes proving Theorem  \ref{thm:main-dual}, which states that $P^nF$, the holim of the iterated $\T^nF$'s as we have defined them, is in fact the $n$-co-excisive approximation to a functor $F$ (when $F$ takes values in Spectra).

We first point out that the original form of the dual calculus and results derived therefrom may be found in \cite{Dual, kuhn-tate, bauervanishing}. A dual tower has stages which naturally map to the functor. Some caution should be made in statements about the dual tower. The version of  $\P^n F$ for any functor from spaces to spectra and any $n\geq 0$ found in \cite{Dual} will be contractible on any space with a finite Postnikov tower. For example, it will vanish on $S^1$, though not necessarily on $S^2$.

One may form the Eckmann-Hilton dual of Goodwillie's calculus theory, switching cocartesian to cartesian everywhere. That is, 

\begin{defn}
For $F$ a homotopy endofunctor of spaces, $F$ is \emph{co-excisive} if it takes homotopy pullbacks to homotopy pushouts and \emph{n-co-excisive} if it takes strongly cartesian $(n+1)$-cubes to cocartesian ones. 
\end{defn}

%In this appendix, 
To prove Theorem  \ref{thm:main-dual}, we will first establish the dual of a key lemma needed to show that the approximations do in fact take strongly cartesian cubes to cocartesian ones, and leave further development of this theory and its background to future work. 

The key lemma is the counterpart of \cite[Lemma 1.9]{GC3}, which shows that the map $t^n : \T^n F \ra F$ factors through some co-cartesian cube. In \cite{GC3}, Goodwillie combined the original Lemma with commutativity of finite pullbacks with filtered colimits to conclude that the construction producing $\P_n F$  produces a homotopy limit cube from a strongly cocartesian $(n+1)-$cube. 

However, it is important to note that we cannot always commute finite pushouts with (co)filtered homotopy limits of spaces. We choose currently to resolve the issue of commuting finite pushouts with (co)filtered homotopy limits by restricting to functors taking values in spectra if we need to consider $\P^n F$. Then
these approximations $\P^nF$ do take strongly cartesian cubes to cocartesian ones, as we will show.

%Formally dualize, take pullbacks to pushouts. Need to show that one lemma. 
%To save on notational headaches, rather than denote by $\T^n$ the dual of $\T_n$ (with iterates $(\T^n)^kF$), we will let $\T^n F$ denote the dual. 
%To define $\T^n$, rather than worry about laying out exactly what the cojoin operation should be, %
%we will let 
Recall that $\T^nF$ is given by our dualization of the functors we use to decompose $T_n$. %
That is, we (in Definition \ref{defn:dualadj}) let $\T^n F:= L^n  F  R^n$ and $(\T^n)^k F := (L^n)^k F  (R^n)^k$. 

\noindent Recall (from section \ref{sec:dualproofs}) that $R^n (X)$ is 
\[
X \mapsto \left( [n]-S \mapsto \lim \left( \begin{array}{c}\xymatrix{ & \ast \ar[d]\\ X^{\Delta^{S}} \ar[r]& \prod \limits_{s\in S} X} \end{array}\right)\right)
\]

\noindent and $L^n$ is the inclusion followed by hocolim.

The adjunctions are between these categories as follows:
\[
\xymatrix{
R^n: \utilde{\text{Fun}}(\mathscr{P}^1 ([n]), \Top) \ar[r]_-{inc}& \ar@<-5 pt>[l]_{cor}  \Fun (\mathscr{P}^1([n]), \Top) \ar[rr]_-{\hocolim} && \ar@<-5 pt>[ll]_-{[n]-S \mapsto X^{\Delta^S}} \Top :L^n \\
}
\]

We then construct the $n$-co-excisive approximation,  %have 
\[
\P^n F(X) := \holim (\cdots\ra(\T^n)^2 F(X)\ra \T^n F(X)).
\]
%We denote the n-co-excisive approximation by $\P^n F$ to keep with the established literature. %, and because $\P^n \P^n = \P^n$, so there will be no need to keep track of iterations of this functor. 

To show that this is an $n$-co-excisive approximation to $F$, we first need to show that it takes strongly cartesian diagrams to cartesian ones (i.e. Theorem \ref{thm:main-dual}). We do this by dualizing the proof which was provided by Charles Rezk \cite{rezk} of Lemma 1.9 of \cite{GC3}.

\begin{prop}(Dual of \cite[Lemma 1.9]{GC3})\label{lem:ncoexc}
 Let $\mathcal{X}$ be any strongly cartesian $(n+1)$-cube and $F$ be any homotopy functor. The map of cubes $t^n F(\mathcal{X}): \T^n F(\mathcal{X}) \ra F(\mathcal{X})$ factors through some cocartesian cube.
\end{prop}

\noindent which we use to show\\

%%%%%%%%%%%%%
%\noindent \textbf{Theorem
\begin{thm} \label{thm:main-dual}
%\textit{
With our definitions as in \ref{defn:dualadj}, and for functors $F$ taking values in Spectra, the functor given by %}
\[
\P^n F(X) := \holim (\cdots\ra(\T^n)^2 F(X)\ra \T^n F(X)).
\]
%\textit{
is $n$-co-excisive. In the homotopy category, $p^n F: \P^n F \ra F$, induced by the map $t^n$ and its iterates, is the universal map to F from an $n$-co-excisive functor.%}
\end{thm}
%%%%%%%%%%%%

%%%%%%%%%%%%%%%%%%%%%%%%
\subsection{Proof of Prop \ref{lem:ncoexc}, Dual of \cite[Lemma 1.9]{GC3}}%)\label{lem:ncoexc}

To prove Prop \ref{lem:ncoexc}, we first need some setup and a lemma. 

Let $U \in [n]-S$ and define for each $U$ the cube $X^U$ as follows
\[
\mathcal{X}^U([n]-S) := \holim 
\left( 
\begin{array}{c}
\xymatrix{
 					& \prod \limits_{u \in U}\mathcal{X}([n]-S-\{u\}) \ar[d]\\
\mathcal{X}( [n]-S) \ar[r] &  \prod \limits_{u \in U}\mathcal{X}([n]-S)
}
\end{array}
\right)
\]

%%%%%%%%%%%%%%%%%%%
\begin{lem}\label{lem:mini}
If $\mathcal{X}$ is strongly cartesian, then $\mathcal{X}^U ([n]-S) \simeq \mathcal{X}([n]-S-U)$. 
\end{lem}
%%%%%%%%%%%%%%%%%%%

\begin{proof}
Since strongly cartesian is a property of the sub-2-faces, we will show this for an arbitrary sub-2-face of $\mathcal{X}$. Let $U=\{u_1, u_2\}$. Strongly cartesianness implies that the following is a homotopy pullback square

%%%%%%%%%
%\begin{figure}[h!]
\[
\mathcal{X}([n]-S-\{u_1,u_2\}) =
\holim 
\scalebox{.85}{$\left(
\begin{array}{c}
\xymatrix{
%\ar[r]\ar[d] 
& \mathcal{X}([n]-S-\{u_1\})\ar[d]\\
\mathcal{X}([n]-S-\{u_2\})\ar[r] & \mathcal{X}([n]-S)\\
}
\end{array}
\right)$}
\]
%\caption{Definition of $\mathcal{X}([n]-S-\{u_1,u_2\}) $}
%\label{fig:lem-mini-1}
%\end{figure}
%%%%%%%%%

\noindent where we take as model for the holim the space of maps from $\Delta_{\leq 1}\circ c_1$ into this diagram, which is the same model as we used previously in this paper for the holim of a punctured square. %That is, the holim will be the space 
%\[
%\{(a, \gamma, b) \in \mathcal{X}([n]-S-\{u_1\}) \x \mathcal{X}([n]-S-)\x\mathcal{X}([n]-S-\{u_2\}) |  \gamma(0)=a \; \& \; \gamma(1) = b \}.
%\]

We will show that $\mathcal{X}^{\{u_1, u_2\}}([n]-S) \simeq  \mathcal{X}([n]-S-\{u_1,u_2\})$.

\noindent We have that
%$\mathcal{X}^{\{u_1, u_2\}}([n]-S)$ is the homotopy limit of the following diagram:

%%%%%%%%%
%\begin{figure}[h!]
\[
\mathcal{X}^{\{u_1, u_2\}}([n]-S)=
\holim 
\scalebox{.85}{$
\left(
\begin{array}{c}
\xymatrix{
 					& \mathcal{X}([n]-S-\{u_1\})\x  \mathcal{X}([n]-S-\{u_2\}) \ar[d]\\
\mathcal{X}( [n]-S) \ar[r] &  \mathcal{X}([n]-S)_{\{u_1\}}\x \mathcal{X}([n]-S)_{\{u_2\}}
}
\end{array}
\right)$}
\]
%\caption{Definition of $\mathcal{X}^{\{u_1, u_2\}}([n]-S)$}
%\label{fig:lem-mini-2}
%\end{figure}
%%%%%%%%%

\noindent where the horizontal map %in Figure \ref{fig:lem-mini-2}
 is the diagonal map and the vertical map is the inclusion of $ \mathcal{X}([n]-S-\{u_i\})$ into the $\{u_i\}-$indexed copy of $ \mathcal{X}([n]-S)$, with index denoted by subscript.

 Let us examine a point in the homotopy pullback, keeping in mind that a map into a product is determined by a map into each factor. An element of the homotopy pullback is the following data
 
 \begin{spacing}{1.25}
 \[
 \begin{array}{ll}
 x\in \mathcal{X}([n]-S) \\
 (y,z) \in  \mathcal{X}([n]-S-\{u_1\})\x  \mathcal{X}([n]-S-\{u_2\})\\
 (y', z') = \textrm{img}(y,z) \textrm{ in } \mathcal{X}([n]-S)_{\{u_1\}}\x \mathcal{X}([n]-S)_{\{u_2\}}\\
\gamma: I \ra  \mathcal{X}([n]-S)_{\{u_1\}}\x \mathcal{X}([n]-S)_{\{u_2\}}\\
 \end{array}
 \]
 \end{spacing}
 
where $\gamma$ may be expressed as a path in each coordinate, $\gamma = (\gamma_1, \gamma_2)$ such that 

\[
 \begin{array}{lll}
\gamma_1: I \ra \mathcal{X}([n]-S)_{\{u_1\}} & \gamma_1 (0) = y' &\gamma_1(1)=x\\
\gamma_2: I \ra \mathcal{X}([n]-S)_{\{u_2\}} & \gamma_2 (0) = x &\gamma_2(1)=z'\\
 \end{array}
 \]
 
% along with a map $\gamma$ of the interval into $\mathcal{X}([n]-S)_{\{u_1\}}\x \mathcal{X}([n]-S)_{\{u_2\}}$, which is actually a map %of the interval into each copy, call them $\gamma_1$ and $\gamma_2$, such that $\gamma_1$ has endpoints $a,b$ and $\gamma_2$ has %endpoints $a,c$.  
 The point $x$ was then effectively superfluous.  Note that we now have $\widetilde{\gamma}$ as the concatenation $\gamma_1 \join \gamma_2$,  between $y'$ and $z'$ in $\mathcal{X}([n]-S)$ This yields the corresponding point in the homotopy limit given in the definition of $\mathcal{X}([n]-S-\{u_1,u_2\}) $, which was the first diagram in this proof. There is a clear (up to homotopy) inverse to this process, and we conclude that the holims are equivalent.
\end{proof}

%%%%%%%%%%%%%%
\begin{proof}[Proof of Lemma \ref{lem:ncoexc}] 
Given Lemma \ref{lem:mini},  we now point out how the map $t^n F$ factors through the cube $X^U$ discussed in the previous Lemma and why this cube will be cocartesian.
%\begin{enumerate}
%(1)
%\item How the map $t^n F$ factors through this cube:\\

Rezk\cite{rezk} observes that there is a natural map $\mathcal{X}_U(S) \ra \mathcal{X}(S) \join U=L_n(\mathcal{X}(S))(U)$ which induces the factorization 
\[
t_n F(\mathcal{X}(S)) : F(\mathcal{X}(S)) \ra \holim \limits_{U \in \mathscr{P}_0([n])} F(\mathcal{X}_U(S)) \ra \T_n F(\mathcal{X}(S)).
\]
We exhibit the dual as a natural map $R^n(\mathcal{X}([n]-S))(U) \ra \mathcal{X}^U([n]-S)$, inducing a factorization
\[
\T^n F(\mathcal{X}([n]-S)) \ra  %
\hocolim_{U \in \mathscr{P}^1([n])}F(\mathcal{X}^U([n]-S)) %
 \ra F(\mathcal{X}([n]-S)).
\]

We provide the map after recalling the two objects involved:

\[
\scalebox{.85}{$
R^n(\mathcal{X}([n]-S))(U) =
\holim 
\left(
\begin{array}{c}
\xymatrix{
 & \ast \ar[d]\\
\mathcal{X}([n]-S)^{\Delta^U} \ar[r] & \prod_{u \in U} \mathcal{X} ([n]-S)
}
\end{array}
\right)$}
\]
and

\[
%
%\;\;
\scalebox{.85}{$
\mathcal{X}^U([n]-S) = \holim 
\left( 
\begin{array}{c}
\xymatrix{
 					& \prod \limits_{u \in U}\mathcal{X}([n]-S-\{u\}) \ar[d]\\
\mathcal{X}( [n]-S) \ar[r] &  \prod \limits_{u \in U}\mathcal{X}([n]-S)
}
\end{array}
\right)$}
\]

\noindent Comparing the diagrams, we note that $\mathcal{X}([n]-S)^{\Delta^U}\ra \prod_{u \in U} \mathcal{X} ([n]-S)$ is a fibrant replacement of the diagonal %$\mathcal{X}([n]-S) \ra  \prod_{u \in U} \mathcal{X} ([n]-S)$ 
and factors naturally through $\mathcal{X}([n]-S)$ as a result.

The map $\ast \ra \prod_{u \in U} \mathcal{X} ([n]-S)$, as before, is the map to the basepoint. This factors through  $\prod \limits_{u \in U}\mathcal{X}([n]-S-\{u\})$.

%%%%%%%%%% begin commented out
\if false
and 
\[
\mathcal{X}^U([n]-S) := \holim 
\left( 
\begin{array}{c}
\xymatrix{
 					& \prod \limits_{u \in U}\mathcal{X}([n]-S-\{u\}) \ar[d]\\
\mathcal{X}( [n]-S) \ar[r] &  \prod \limits_{u \in U}\mathcal{X}([n]-S)
}
\end{array}
\right)
\]
\fi
%%%%%%%%%% \Commented out
%[[Map should be clear at this point, but spell it out anyway]]
\bigskip

%\noindent (2) 
%\item Why this cube will be cocartesian:\\ %(same reason as in the 'normal' case):\\
%((draft 1, rewrite))\\

We can consider $\mathcal{X}^U$ as two sub-cubes which differ by an element $\{u\} \in U$, we have that the maps $\mathcal{X}([n]-U-\{u\}) \ra \mathcal{X}([n]-U)$ are isomorphisms; for nonempty $U$, the cube is cocartesian.
%\end{enumerate}
\end{proof}
%%%%%%%%%%%%%%

%%%%%%%%%%%%%%%%%%%%%%%%
\subsection{Proof of Theorem \ref{thm:main-dual}}
There are three parts of this proof. First, that $\P^nF$ is actually $n$-co-excisive.  Then the existence and uniqueness of a map which `co' factors a map $v: P \ra F$ for $P$ some co-$n$-excisive functor. 

%%%%%%%%%%%%%%%%%%%%%%%%
\subsubsection{Co-$n$-excisiveness} 
In \cite{GC3}, Lemma 1.9, the counterpart of Prop \ref{lem:ncoexc}, was then combined with commutativity of finite pullbacks with filtered colimits to conclude that applying $\hocolim (\T_n F \ra \T_n^2 F\ra \cdots)$ to a strongly cocartesian $(n+1)-$cube produced a cartesian cube. 

%The dual situation does not always happen. That is, w
We cannot always commute finite pushouts with (co)filtered homotopy limits of spaces. Since our current aim is not a complete re-write of the dual calculus theory to endofunctors of spaces, we choose to resolve the issue of commuting finite pushouts with (co)filtered homotopy limits by restricting to functors taking values in spectra. 

Let $\X$ be a strongly cartesian $(n+1)$-cube. By Prop \ref{lem:ncoexc}, each of the maps of  the holim defining $P^nF\X$, 
\[
\holim ( \T^n F\X \ra (\T^n)^2F\X \ra \cdots),
\]
factors through some cocartesian cube. Then $\P^n F\X$ is equivalent to this sequential holim of cocartesian cubes.  Since pushouts and pullbacks agree in spectra, we commute finite pushouts with (co)filtered homotopy limits and conclude that the holim of cocartesian cubes is again cocartesian. That is, $\P^n F$ is $n$-co-excisive. 
%is itself cocartesian, as we land in spectra.  That is, that the functor defined as $\holim ( \T^n F \ra (\T^n)^2F \ra \cdots)$ is equivalent to a holim of cocartesian cubes (thanks to Prop\ref{lem:ncoexc}). Landing in spectra allows us to commute 
%\footnote{}.

%%%%%%%%%%%%%%%%%%%%%%%%
\subsubsection{Existence of a co-factorization} %map which `co' factors a map $v: P \ra F$ for $P$ a random co-$n$-excisive functor. 

We follow the proof in \cite{GC3}.  We first show uniqueness in a similar way. Let $P$ be some $n$-co-excisive functor and $P\overset{u}{\ra} F$ a weak map (a zig zag of maps is a``weak map"; it is a map in the homotopy category).  We then have a commutative square 

\[
\xymatrix{
\P^n F \ar[r]^{\P^n u} \ar[d]_{p^n P} & \P^n F \ar[d]^{p^n F}\\
P \ar[r]^{u} & F&.\\
}
\] 

Due to $n$-co-excisiveness of $P$, we get that $p^nP$ is invertible as a weak map, giving us our (in the homotopy category) co-factorization of $u$, re-writing the above square while taking into account this invertability:

\[
\xymatrix{
%\P^n F \ar[r]^{\P^n u} \ar[d]_{p^n P} 
                                            & \P^n F \ar[d]^{p^n F}\\
\P^n F \simeq \ar[ur]P \ar[r]^{u} & F&.\\
}
\]

%%%%%%%%%%%%%%%%%%%%%%%%
\subsubsection{Uniqueness of a co-factorization} %map which `co' factors a map $v: P \ra F$ for $P$ a random co-$n$-excisive functor. 

We need to show that if $P$ is $n$-co-excisive, then a weak map $v: P \ra P^n F$ is determined by the composition $p^n F \circ v$ (that is, comes from a weak map $P \ra F$).% (order of maps?). 

It suffices to show that in the following diagram of weak maps, those labeled with $\sim$ are in fact invertible,

\[
\xymatrix{
P^n P \ar[r]^{P^n v} \ar[d]^{\sim}_{p^n P} & P^n P^n F \ar[r]^{P^np^n F}_{\sim}\ar[d]^{p^n P^n F}_{\sim} & P^n F \ar[d]^{p^n P}\\
P \ar[r]^v  &P^n F \ar[r]^{p^n F} & F
}.
\]

Given the invertibility of these above maps,  $v$ is then determined by $P^n v$, which is determined by $P^np^n F \circ P^n v = P^n (p^n F \circ v)$, which is clearly determined by $p^n F \circ v$. 

Since $P$ and $P^nF$ are $n$-co-excisive, the vertical marked weak maps are invertible. For the remaining map $P^n (p^n F)$ to be an equivalence, it is sufficient for $P^n (t^n F)$ to be an equivalence, as $p^n$ is the map induced by taking the limit other the iterations of $t^n$). Then 
\[
\xymatrix{
P^n F & P^n T^n F \ar[l]_-{P^n (t^n F)} := P^n (L^n FR^n)
}
\]
Using that our functors take values in spectra, since $L^n$ is a finite hocolim and in spectra, holims commute with finite hocolims, we can pull $P^n$ past $L^n$
\[
P^n (L^n FR^n) \simeq L^n P^n F R^n
\]
and as $P^n F$ is $n$-co-excisive, it takes the strongly cartesian cube that $R^n$ outputs to a cocartesian one. That is, the composition 
\[
\xymatrix{
P^n F & &P^n T^n F \ar[ll]_-{P^n (t^n F)} := P^n (L^n FR^n)\simeq L^n P^n F R^n
}
\]
is an equivalence. That is, $P^n(t^n F)$ is an equivalence. 

%%%%%%%%%%%%%%%%%%%%%%%%%%%%%%%%%%%%%%
%%%%%%%%%%%%%%%%%%%%%%%%%%%%%%%%%%%%%%
%
\section{LS category and related corollaries} \label{sec:lscat}
%
%
%%%%%%%%%%%%%%%%%%%%%%%%%%%%%%%%%%%%%%
%%%%%%%%%%%%%%%%%%%%%%%%%%%%%%%%%%%%%%

%\input{sec-LScat}
This section provides further discussion of the relationship between LScategory and the constructions of Goodwillie calculus, including proofs and more details around the corollaries of Theorem \ref{thm:main-dual}.

%\subsection{Versions and equivalence of the notions of category}
There are three major equivalent notions of LS category for a space, those of Ganea, Whitehead and Hopkins. A contemporary proof of this equivalence may be found in  \cite[\S V Ch. 27]{felix}. Doeraene \cite[Theorem 3.11]{doeraene93} was the first to show 
the equivalence of the Ganea and Whitehead definitions. Hopkins in \cite{hopkins} establishes the equivalence of Ganea's and his notions.
%), but mention earlier work to emphasize what fails when trying to dualize. 

% \subsection{Hopkins symmetric LScategory}

%In section we gave the explanation of how to translate between Hopkins symmetric LScocat and our calculus language. 
This definition of symmetric LScategory \cite[Section 3, p221]{hopkins} is formally dual to Hopkins' definition of symmetric LScocat. He defined, for a given space $X$, a contravariant functor $F_n$, as the homotopy colimit of a co-punctured $(n+1)$ cube whose $A$-indexed position is homotopic to $\prod_{|A|} \Omega X$.  For each $n$, his $F_n$ is then our $\T^n \I$.  He constructs a directed system\footnote{Hopkins states \cite[p.91]{hopkinsthesis} that this $F_n$ sequence can be identified with the Milnor filtration of $X$ regarded as the classifying space of its loop space. } of these $F_n$, with the maps cofibrations: 
 
\[
 \hocolim F_1 \ra  \hocolim F_2 \ra \cdots \ra \hocolim F_n \ra \cdots .
%(\cdots \ra \holim F^n \ra \holim F^{n-1} \ra \cdots)
\]
%which is therefore our $\T_n \I$ tower, 
% \[
% (\cdots \ra \T_n\I(X) \ra \T_{n-1}\I(X) \ra \cdots \T_1 \I(X)).
% \]

Which is then the first co-tower for the dual calculus as we have defined it. 
 
% \noindent Theorem 3.2.2 of \cite{hopkinsthesis} is that the homotopy inverse limit of a construction that is equivalent to the  tower of $F^n$'s gives $\Z_\infty X$ when $X$ is connected, i.e.  that $\holim_n \T_n \I(X) \sim \Z_\infty X$. 

%%%%%%%%%%%%%%%%%%%%%%%%%%%%%%%%
\subsection{Consequences and corollaries} 
%%%%%%%%%%%%%%%%%%%%%%%%%%%%%%%%
Returning to our functor-calculus language and recalling that we defined $\T^nF$ as $L^n F R^n$ (Definition \ref{defn:dualadj}) where $L^n, R^n$ are natural dualizations of $L_n, R_n$, we re-state Hopkins's relevant definitions as %\\

%%%%%%%%
\begin{prop}\label{defn:symcat}
For a space $X$, symmetric LS cat(X) $\leq n$ if and only if  the natural map $\T^n\I (X)\ra X$ has a section up to homotopy.
\end{prop}
%%%%%%%%
%
%%%%%%%%%
%\noindent \textbf{Definition \ref{defn:symcat}} \textit{For a space $X$, symmetric LS cat(X) $\leq n$ if and only if  the natural map $\T^n\I (X)\ra X$ has a section (up to homotopy). }\\
%%\end{defn}
%%%%%%%%%

As we are \textit{defining} $\T^n$ this way, we did not strictly dualize Theorem \ref{thm:main}. %The related result (Theorem \ref{thm:main-dual}) -- that this definition of $\T^n$ leads to a $n$-co-excisive approximation -- is left for Appendix \ref{sec:bkDGC}.
However, given our above definitions,we have the following corollaries as consequences, which are the duals of those in Section \ref{sec:lscocat}. %\\

%Which has as a consequence 

%%%%%%%%
%\noindent \textbf{Corollary \ref{cor:dualTnI}}
\begin{cor}\label{cor:dualTnI}
%\textit{
$\T^nF$ are left and right $\T^n \I$-functors, as are the $\P^nF$.%}\\
\end{cor}
%%%%%%%%

\begin{proof} As with the analogous result (\ref{cor:TnI}), we establishing the result for $\T^nF$ and iterates implies it for the $\P^nF$ since a limit of objects which have sections to the same object also has such sections.

This corollary follows immediately from combining the definition of $\T^nF$ (Definition \ref{defn:dualadj}) with Proposition \ref{prop:RFL}.
\end{proof}

The remaining dual corollaries we only state for the iterated $\T^n$ constructions, since we have only dealt with $\P^n$ for functors which take values in spectra, and these results only make sense for functors which take values in spaces. Also keep in mind that as one iterates the $\T^n$'s, the input should have higher and higher connectivity: e.g. $(T^1)^kF = \Sigma^k F \Omega^k$. 

Given the corollary,  as before, this extra structures implies that\\
 %e structure of being a co-bimodule leads us immediately to 

%%%%%%%%
%\noindent \textbf{Corollary \ref{cor:dualsym}}
\begin{cor}\label{cor:dualsym}
%\textit{
$\T^nF$ takes values in spaces of symmetric LS cat $\leq n$, as do the higher iterates $(\T^n)^k F$.%}\\ % $\P^nF$.   
\end{cor}
%\end{cor}
%%%%%%%%

This proof is formally dual to the proof of Corollary \ref{cor:TnI}; we leave it for the interested reader. 

%\if false
%\begin{proof}
%This follows immediately from the retract structure established in Corollary \ref{cor:TnI}. For $X$ a space, there is always a map $X=\I(X) \ra \T_n \I (X)$.   So also for a space $\T_n F(X)$, we have a map $\T_n F(X) \ra \T_n \I \circ \T_n F(X)$. Writing in the adjunctions, we have the following: 
%\[
%\begin{array}{ccc}
%\T_n FX  &\lra & \T_n \I \circ \T_n F(X)\\
%\rcong & & \rcong\\
%R_n F L_n & \lra   & R_n L_n R_n F L_nX\\
%\end{array}
%\]
%
%The counit of the adjunction, $\epsilon: L_n R_n \I \ra \I$, gives us our map 
%\[
%R_n L_n R_n F L_nX \ra R_n F L_n
%\]
%Recall that the counit and the unit of the adjunction give that the following composition
%\[
%\xymatrix{
%%L_n \ar[r]^{L_n \eta}& L_n R_n L_n \ar[r]^{\epsilon L_n} & L_n\\
%R_n \ar[r]^{\eta R_n}& R_n L_n R_n \ar[r]^{R_n \epsilon} & R_n\\
%}
%\]
%is the identity. 
%
%We can apply this to $FL_n X$ to realize it as our maps 
%\[
%%\xymatrix{
%R_n F L_nX \ra R_n L_n R_n F L_nX \ra R_n F L_n
%%}
%\]
%[REALLY?:] however, our first map is, at best, only homotopic to $\eta R_n \circ FL_n X$. 
%\end{proof}
%%Due to the following 
%%%%%%%%%
%\fi 

%%%%%%%%
%we conclude that 
 
 %%%%%%%%
\begin{cor}\label{cor:CL}
%\noindent \textbf{Corollary \ref{cor:CL}}
%\textit{
The cup products of length $\geq n+1$ vanish for $\T^n F(X)$ and the higher iterates $(\T^n)^k F$.%}\\ %$\P^n F(X)$ with any space $X$. 
\end{cor}
%%%%%%%%
\begin{proof}	
%Combining \ref{ineq-nil} with 
%Corollary \ref{cor:sym} tells us that $\T_nF$ and $\P_n$ take values in spaces %of LS cocat $\leq n$. 
Similar to Whitehead length, we define cup-length to be the maximum length of non-trivial cup products (minus 1) in $H^\ast X$.  Combining Corollary \ref{cor:dualsym} with the inequality  %chain of inequalities, 
%\begin{center}
%cup-length (X)  $\leq$ conil ($\Sigma X$) $\leq$ LS cat(X)
cup-length (X)  $\leq$ LS cat(X) of \cite{BersteinGanea},
%\end{center}
we conclude our result. 
\end{proof}

%%%%%%%%%%%%%%%%%%%%%%%%%%%%%%%%%%%%
%%%%%%%%%%%%%%%%%%%%%%%%%%%%%%%%%%%%
%%%%%%%%%%%%%%%%%%%%%%%%%%%%%%%%%%%% 
\addcontentsline{toc}{part}{Bibliography}
\bibliographystyle{amsalpha}
\bibliography{LScocatbib.bib}

\providecommand{\bysame}{\leavevmode\hbox to3em{\hrulefill}\thinspace}
\providecommand{\MR}{\relax\ifhmode\unskip\space\fi MR }
\providecommand{\MRhref}[2]{%
  \href{http://www.ams.org/mathscinet-getitem?mr=#1}{#2}
}
\providecommand{\href}[2]{#2}
\begin{thebibliography}{DHKS05}

\bibitem[BD]{unBiedDwy}
G.~Biedermann and W.G. Dwyer, \emph{Homotopy nilpotent spaces}, in preparation.

\bibitem[BD10]{Bied-Dwy}
\bysame, \emph{Homotopy nilpotent groups}, Algebraic {\&} {G}eometric
  {T}opology \textbf{10} (2010), no.~1, 33--61.

\bibitem[BEJM15]{BEJM14}
K.~Bauer, R.~Eldred, B.~Johnson, and R.~McCarthy, \emph{{G}eometric models for
  {T}aylor polynomials of functors}, arxiv:1506.02112v1, 2015.

\bibitem[BG61]{BersteinGanea}
I~Berstein and T~Ganea, \emph{Homotopical nilpotency}, Illinois journal of
  {M}athematics \textbf{5} (1961), no.~1, 99--130.

\bibitem[BJM15]{BJM14}
K.~Bauer, B.~Johnson, and R.~McCarthy, \emph{A {C}otriple {M}odel for
  {C}alculus in an {U}nbased {S}etting (with an {A}ppendix by {R}osona {E}ldred
  )}, Transact. AMS \textbf{367} (2015), 6671--6718.

\bibitem[BK72]{BK}
A.K. Bousfield and D.M. Kan, \emph{Homotopy limits, completions and
  localizations}, Springer, 1972.

\bibitem[BM04]{bauervanishing}
K.~Bauer and R.~McCarthy, \emph{On vanishing {T}ate cohomology and
  decompositions in {G}oodwillie calculus}, preprint version (2004).

\bibitem[CS12]{chorny-scherer}
B.~Chorny and J.~Scherer, \emph{Goodwillie calculus and {W}hitehead products},
  To appear in {F}orum {M}athematicum. Arxiv preprint math/1109.2691. DOI
  10.1515/forum-2012-0038, 2012.

\bibitem[CSV15]{csv}
C.~Costoya, J.~Scherer, and A.~Viruel, \emph{A torus theorem for homotopy
  nilpotent groups}, arxiv: 1504.06100, 2015.

\bibitem[Del00]{comonads}
A.~Deligiannis, \emph{Ganea comonads}, Manuscripta {M}athematica \textbf{102}
  (2000), no.~2, 251--261.

\bibitem[DHKS05]{Dwyerhirschetc}
W.G. Dwyer, P.S. Hirschhorn, D.M. Kan, and J.H. Smith, \emph{Homotopy limit
  functors on model categories and homotopical categories}, vol. 113, Amer
  Mathematical Society, 2005.

\bibitem[DI04]{dugger-appendix}
D.~Dugger and D.C. Isaksen, \emph{Topological hypercovers and
  $a^1$-realizations}, Mathematische Zeitschrift \textbf{246} (2004), no.~4,
  667--689.

\bibitem[Doe93]{doeraene93}
J.~Doeraene, \emph{{LS}-category in a model category}, Journal of {P}ure and
  {A}pplied {A}lgebra \textbf{84} (1993), no.~3, 215--261.

\bibitem[Dug08]{duggerPrimer}
D.~Dugger, \emph{A primer on homotopy colimits}, Preprint, 2008.

\bibitem[Eld13]{me-AGT}
R.~Eldred, \emph{Cosimplicial models for the limit of the {G}oodwillie tower},
  Algebraic and {G}eometric {T}opology \textbf{13} (2013), 1161--1182.

\bibitem[FHT01]{felix}
Y.~F{\'e}lix, S.~Halperin, and J.C. Thomas, \emph{Rational homotopy theory},
  vol. 205, Springer Verlag, 2001.

\bibitem[Gan60]{GaneaLS}
T.~Ganea, \emph{{L}usternik-{S}chnirelmann category and cocategory},
  Proceedings of the {L}ondon {M}athematical {S}ociety \textbf{3} (1960),
  no.~1, 623.

\bibitem[Goo90]{GC1}
T.~Goodwillie, \emph{Calculus {I}: The first derivative of pseudoisotopy
  theory}, K-theory (1990), 1--27.

\bibitem[Goo91]{GC2}
\bysame, \emph{Calculus {II}: analytic functors}, K-theory (1991), 295--332.

\bibitem[Goo03]{GC3}
\bysame, \emph{Calculus {III}: {T}aylor series}, Geometry and {T}opology
  (2003), 645--711.

\bibitem[Hop84a]{hopkins}
M~J Hopkins, \emph{Formulations of cocategory and the iterated suspension},
  Ast\'{e}risque \textbf{113-114} (1984), 212--226, Algebraic homotopy and
  local algebra ({L}uminy, 1982).

\bibitem[Hop84b]{hopkinsthesis}
\bysame, \emph{Some problems in topology}, Ph.D. thesis, Oxford {U}niversity,
  1984.

\bibitem[JM04]{JM-cotriples}
B.~Johnson and R.~McCarthy, \emph{Deriving calculus with cotriples},
  Transactions of the {A}merican {M}athematical {S}ociety \textbf{356} (2004),
  no.~2, 757--803.

\bibitem[Kuh04]{kuhn-tate}
N.J. Kuhn, \emph{Tate cohomology and periodic localization of polynomial
  functors}, Inventiones mathematicae \textbf{157} (2004), no.~2, 345--370.

\bibitem[LS34]{LS}
L.~Lusternik and L.~Schnirelmann, \emph{Topological methods in the calculus of
  variations}, Actualit\'{e}s scientifiques et industrielles (1934).

\bibitem[May72]{May}
J.~P. May, \emph{The {G}eometry of {I}terated {L}oop {S}paces}, Lecture {N}otes
  in {M}athematics, vol. 271, Springer-Verlag, New York, 1972.

\bibitem[McC01]{Dual}
R.~McCarthy, \emph{Dual calculus for functors to spectra}, Contemporary
  {M}athematics \textbf{271} (2001), 183--216.

\bibitem[ML98]{maclane}
S.~Mac~Lane, \emph{Categories for the working mathematician}, vol.~5, Springer
  verlag, 1998.

\bibitem[Rez13]{rezk}
C.~Rezk, \emph{A streamlined proof of {G}oodwillie's n-excisive approximation},
  Algebraic and {G}eometric {T}opology \textbf{13} (2013), no.~2, 1049--1051.

\end{thebibliography}
%%%%%%%%%%%%%%%%%%%%%%%%%%%%%%%%%%%%
%%%%%%%%%%%%%%%%%%%%%%%%%%%%%%%%%%%%
%%%%%%%%%%%%%%%%%%%%%%%%%%%%%%%%%%%%

\end{document}